\documentclass[11pt]{article}
\usepackage{e-jc}
\usepackage{amsthm,amsmath,amssymb}

\usepackage[colorlinks=true,citecolor=blue,linkcolor=black,urlcolor=blue]{hyperref}

\usepackage{graphicx}

\numberwithin{equation}{section}
\numberwithin{figure}{section}
\theoremstyle{plain}
\newtheorem{thm}{Theorem}[section]
\newtheorem{lem}[thm]{Lemma}
\newtheorem{cor}[thm]{Corollary}
\newtheorem{prop}[thm]{Proposition}

\theoremstyle{remark}

\newcommand{\M}{\operatorname{M}}
\newcommand{\Hf}{\operatorname{H}}
\newcommand{\wt}{\operatorname{wt}}

\title{A $q$-enumeration of lozenge tilings of a hexagon with three dents}
\author{Tri Lai\footnote{This research was supported in part by the Institute for Mathematics and its Applications with funds provided by the National Science Foundation (grant no. DMS-0931945).}\\
\small Institute for Mathematics and its Applications\\[-0.8ex]
\small University of Minnesota\\[-0.8ex]
\small Minneapolis, MN 55455\\
\small email: \texttt{tmlai@ima.umn.edu}\\
\small website: \url{http://www.ima.umn.edu/~tmlai/}
}

\date{\small Mathematics Subject Classifications: 05A15,  05C30, 05C70}

\begin{document}
\maketitle

\begin{abstract}
We $q$-enumerate lozenge tilings of a hexagon with three bowtie-shaped regions have been removed from three non-consecutive sides. The unweighted version of the result generalizes a problem posed by James Propp on enumeration of lozenge tilings of a hexagon of side-lengths $2n,2n+3,2n,2n+3,2n,2n+3$ (in cyclic order) with the central unit triangles on the $(2n+3)$-sides removed.

\end{abstract}

\section{Introduction}
A \emph{plane partition} is a rectangular array of non-negative integers so that all rows are weakly decreasing from left to right and all columns are weakly decreasing from top to bottom. Plane partitions having $a$ rows and $b$ columns with entries at most $c$  are usually identified with their 3-D interpretations---piles of unit cubes fitting in an $a\times b \times c$ box.  The latter are in bijection with the lozenge tilings of a semi-regular hexagon of side-lengths $a,b,c,a,b,c$ (in clockwise order, starting from the northwest side) on the triangular lattice, denoted by $Hex(a,b,c)$. Here, a \emph{lozenge} is a union of any two unit equilateral triangles sharing an edge; and a \emph{lozenge tiling} of a \emph{region}\footnote{The regions considered in our paper are always finite connected regions on the triangular lattice.} is a covering of the region by lozenges so that there are no gaps or overlaps.  The \emph{volume} (or the \emph{norm}) of the plane partition $\pi$ is defined to be the sum of all its entries, and denoted by $|\pi|$.

Let $q$ be an indeterminate. The \emph{$q$-integer} is defined by $[n]_q:=1+q+q^2+q^{n-1}$. We also define the \emph{$q$-factorial} $[n]_q!:=[1]_q[2]_q\dots[n]_q$, and the \emph{$q$-hyperfactorial} $\Hf_q(n):=[0]_q![1]_q!\dots[n-1]_q!$. MacMahon's classical theorem  \cite{Mac}  states that
\begin{equation}\label{qMacMahon}
\sum_{\pi}q^{|\pi|}=\frac{\Hf_q(a)\Hf_q(b)\Hf_q(c)\Hf_q(a+b+c)}{\Hf_q(a+b)\Hf_q(b+c)\Hf_q(c+a)},
\end{equation}
where the sum of the left-hand side is taken over all plane partitions $\pi$ fitting in an $a\times b \times c$ box.

The $q=1$ specialization of MacMahon's theorem is equivalent to the fact that the number of lozenge tilings of the hexagon $Hex(a,b,c)$ is equal to
\begin{equation}\label{MacMahon}
\frac{\Hf(a)\Hf(b)\Hf(c)\Hf(a+b+c)}{\Hf(a+b)\Hf(b+c)\Hf(c+a)},
\end{equation}
where $\Hf(n)=\Hf_1(n)=0!1!\dotsc(n-1)!$ is the ordinary hyperfactorial.  The theorem of MacMahon inspired a large body of work, focusing on enumeration of  lozenge tilings of hexagons with defects (see e.g. \cite{Ciucu1, Ciucu2, CF14, CF15, CK2, CLP, Eisen, CEKZ, Gessel}, or the references in \cite{Propp,Propp2} for a more extensive list). Put differently, the MacMahon's theorem give a $q$-enumeration of lozenge tilings of a semi-regular hexagon. However, such $q$-enumerations are \emph{rare} in the domain of enumeration of lozenge tilings. Together with the related work \cite{Tri}, this paper gives such a rare $q$-enumeration.

 In 1999, James Propp \cite{Propp} published a list of $32$ open problems in the field of enumeration of tilings (equivalently, perfect matchings). Problem 3 on this list asks for the number of lozenge of tilings of a  hexagon of side-lengths\footnote{From now on, we always list the side-lengths of a hexagon in clockwise order, starting from the northwest side.} $2n+3,2n,2n+3,2n,2n+3,2n$, where three central unit triangles have been removed from its long sides (see Figure \ref{original3dent} for the case when $n=2$).  Theresia Eisenk\"{o}lbl \cite{Eisen} solved and generalized this problem by computing the number of lozenge tilings of a hexagon with side-lengths $a,$ $b+3,$  $c,$
$a+3,$ $b,$ $c+3,$ where an arbitrary unit triangle has been removed from each of  the $(a+3)$-, $(b+3)$- and $(c+3)$-sides. Her proof uses non-intersecting lattice paths, determinants,
and the Desnanot-Jacobi determinant identity \cite[pp. 136-149]{Mui} (sometimes mentioned as \emph{Dodgson condensation} \cite{Dodgson}).

\begin{figure}\centering
\includegraphics[width=6cm]{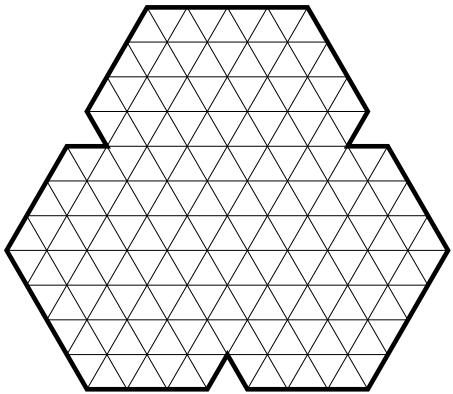}
\caption{A hexagon of side-length $4,7,4,7,4,7$ with three central unit triangles removed from the $7$-sides.}
\label{original3dent}
\end{figure}

One can view the unit triangles removed in the Propp's problem as triangular dent of size 1. In this paper, we consider a more general situation when  our hexagon has three triangular dents  of  \emph{arbitrary} sizes on three non-consecutive sides. Moreover, these triangular dents can be extended to bowtie-shaped dents consisting of  two adjacent triangles as follows.

Assume that  $a,b,c,d,e,f,x,y,z$ are 9 non-negative integers. We consider the hexagon of side-lengths $z+x+a+b+c$, $x+y+d+e+f,$  $y+z+a+b+c,$ $z+x+d+e+f,$ $x+y+a+b+c,$ $y+z+d+e+f$. Next, we remove three bowties along the northeast, south and northwest sides of the hexagon at the locations specified as in Figure \ref{threedents} (for the case of $a=2,b=3,c=2,d=3,e=2,f=2,x=2,y=1,z=2$). We denote by $F\begin{pmatrix}x&y&z\\a&b&c\\d&e&f\end{pmatrix}$ the resulting region.

We call the common vertex of two triangles in a bowtie the \emph{center} of the bowtie. We notice that the centers of  the three bowtie-shaped dents in our $F$-type region are always the vertices of a down-pointing equilateral triangle of side-length  $x+y+z+d+e+f$ (indicated by the dotted triangle in Figure \ref{threedents}).

\begin{figure}\centering
\includegraphics[width=8cm]{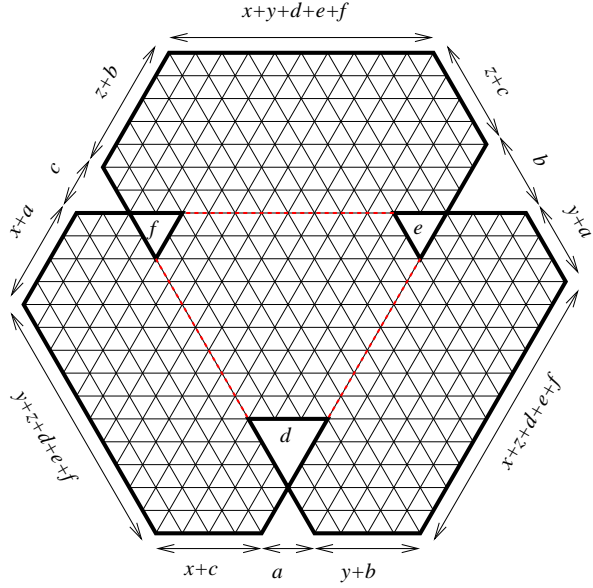}
\caption{A hexagon with three bowtie-shaped dents on three non-consecutive sides.}
\label{threedents}
\end{figure}

\begin{thm}\label{main} For non-negative integers $a,$ $b,$ $c,$ $d,$ $e,$ $f,$ $x,$ $y,$ $z$, the number of lozenge tilings of the region $F\begin{pmatrix}x&y&z\\a&b&c\\d&e&f\end{pmatrix}$ is equal to
\small{\begin{align}\label{maineq}
 &\frac{\Hf(x)\Hf(y)\Hf(z)\Hf(a)^2\Hf(b)^2\Hf(c)^2\Hf(d)\Hf(e)\Hf(f)\Hf(d+e+f+x+y+z)^4}{\Hf(a+d)\Hf(b+e)\Hf(c+f)\Hf(d+e+x+y+z)\Hf(e+f+x+y+z)\Hf(f+d+x+y+z)}\notag\\
 &\times \frac{\Hf(A+2x+2y+2z)\Hf(A+x+y+z)^2}{\Hf(A+2x+y+z)\Hf(A+x+2y+z)\Hf(A+x+y+2z)}\notag\\
  &\times \frac{\Hf(a+b+d+e+x+y+z)\Hf(a+c+d+f+x+y+z)\Hf(b+c+e+f+x+y+z)}{\Hf(a+d+e+f+x+y+z)^2\Hf(b+d+e+f+x+y+z)^2\Hf(c+d+e+f+x+y+z)^2}\notag\\
  &\times\frac{\Hf(a+d+x+y)\Hf(b+e+y+z)\Hf(c+f+z+x)}{\Hf(a+b+y)\Hf(b+c+z)\Hf(c+a+x)}\notag\\
  &\times \frac{\Hf(A-a+x+y+2z)\Hf(A-b+2x+y+z)\Hf(A-c+x+2y+z)}{\Hf(b+c+e+f+x+y+2z)\Hf(c+a+d+f+2x+y+z)\Hf(a+b+d+e+x+2y+z)},
\end{align}}
\normalsize where $A=a+b+c+d+e+f$.
\end{thm}
By letting $x=y=n$, $a=b=c=1$ and $d=e=f=0$ in Theorem \ref{main}, we obtain the  solution of the Propp's problem. Moreover,  the region  $F\begin{pmatrix}x&y&z\\0&0&0\\d&0&0\end{pmatrix}$ has several lozenges on the base that are forced to be in any tiling. By removing these forced  lozenges, we obtain a semi-regular hexagon $Hex(z+x,x+y+d,y+z)$ that has the same number of tilings as the original $F$-type region. Therefore, Theorem \ref{main} implies MacMahon's tiling formula (\ref{MacMahon}).

\begin{figure}\centering

\includegraphics[width=15cm]{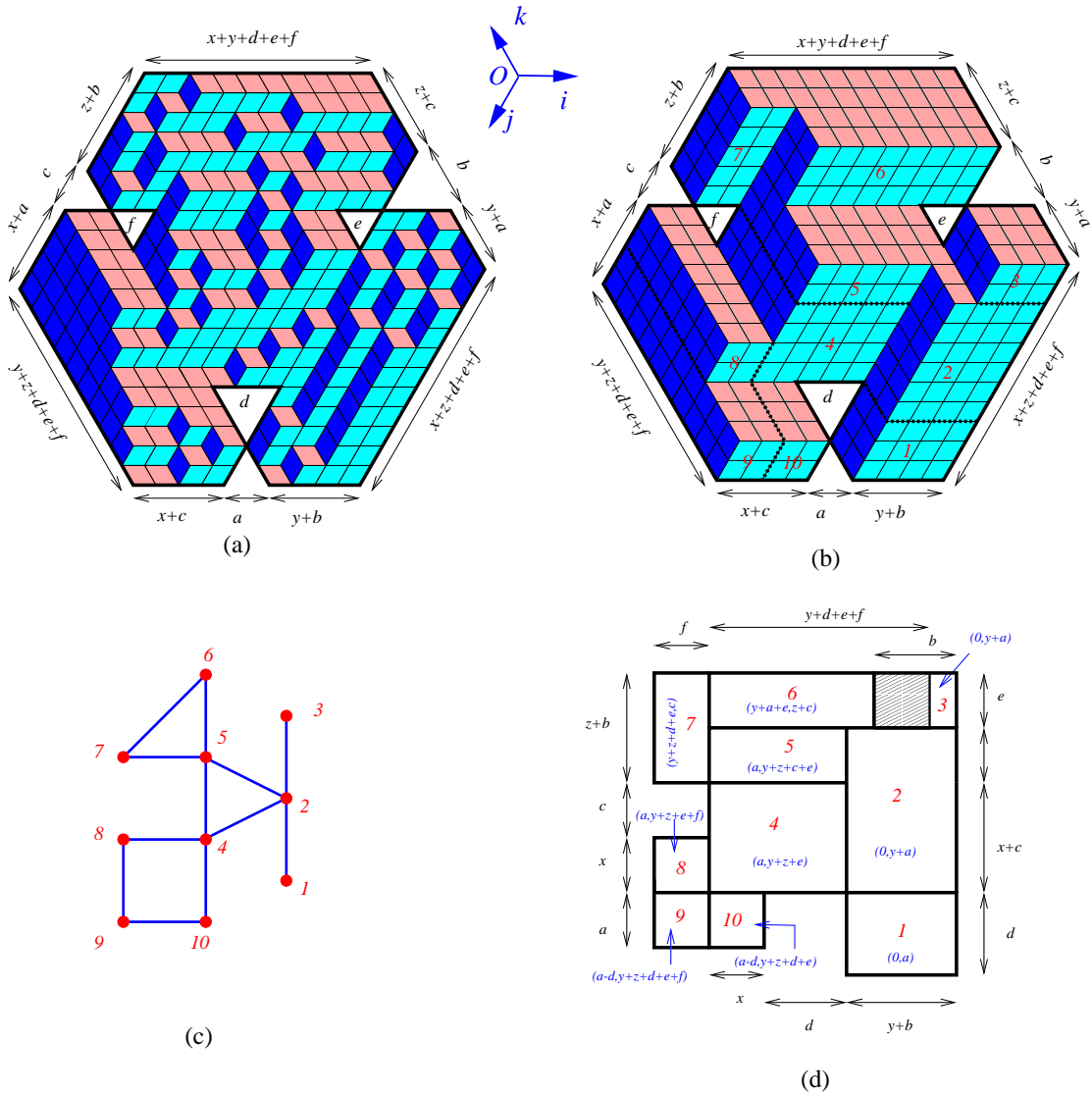}
\caption{(a) Viewing a tiling of a $F$-type region as a pile of cubes in fitting in a compound box $\mathcal{B}$. (b) The empty pile --- a 3-D picture of the compound box $\mathcal{B}$. (c) The connectivity of the component boxes $B_i$'s in $\mathcal{B}$. (d) The projection of the compound box $\mathcal{B}$ on the $\textbf{Oij}$ plane.}
\label{3dentcombine}
\end{figure}

\medskip

Next, we consider a $q$-analogue of Theorem \ref{main} as follows.

Similar to the case of semi-regular hexagons, the lozenge tilings of our $F$-type region can also be viewed as piles of unit cubes fitting in a \emph{compound box} $\mathcal{B}:=\mathcal{B}\begin{pmatrix}x&y&z\\a&b&c\\d&e&f\end{pmatrix}$ consisting of 10 non-overlapping component (rectangular) boxes $B_1,B_2,\dotsc,B_{10}$ (see Figure \ref{3dentcombine}(a)). More precise, Figure \ref{3dentcombine}(b) gives a 3-D picture of the compound box $\mathcal{B}$ by showing the empty pile; the bases of the component boxes $B_i$'s are labelled by $1,2,\dotsc,10$. If two component boxes $B_i$ and $B_j$ are adjacent, we remove the common portion of their faces to make them connected.  The connectivity of the component boxes in $\mathcal{B}$ is illustrated by the graph in Figure \ref{3dentcombine}(c): the vertices are $B_i$'s, and the edges connect precisely two adjacent component boxes. One more way to determine precisely the structure of the compound box $\mathcal{B}$ is to project it on the $\textbf{Oij}$ plane (see Figure \ref{3dentcombine}(d)). In this projection, each component box $B_i$ is represented by the rectangle $i$ associated with a pair of integers $(s,t)$, where $s$ is the level of the base of the box and $t$ is its height. We notice that the rectangles corresponding to $B_3$ and $B_6$ are overlapped (shown by the shaded area in Figure \ref{3dentcombine}(d)), but the boxes themselves are not overlapped (the base of $B_6$ is above the top of $B_3$).

 We call the latter piles  of unit cubes  \emph{generalized plane partitions} (or GPPs), since they satisfy the same monotonicity as the ordinary plane partitions: the tops of the columns (shown as right lozenges in Figure \ref{3dentcombine}(a)) are weakly decreasing along $\overrightarrow{\textbf{Oi}}$ and $\overrightarrow{\textbf{Oj}}$. We also use the notation $|\pi|$ for the volume of the GPP $\pi$.

Similar to MacMahon's classical theorem, we have a simple product formula for the generating function of (the volume of) the GPPs.
\begin{thm}\label{qmain} For non-negative integers $a,$ $b,$ $c,$ $d,$ $e,$ $f,$ $x,$ $y,$ $z$
\small{\begin{align}\label{maineqbox}
& \sum_{\pi}q^{|\pi|}=\notag\\
 &=\frac{\Hf_q(x)\Hf_q(y)\Hf_q(z)\Hf_q(a)^2\Hf_q(b)^2\Hf_q(c)^2\Hf_q(d)\Hf_q(e)\Hf_q(f)\Hf_q(d+e+f+x+y+z)^4}{\Hf_q(a+d)\Hf_q(b+e)\Hf_q(c+f)\Hf_q(d+e+x+y+z)\Hf_q(e+f+x+y+z)\Hf_q(f+d+x+y+z)}\notag\\
 &\times \frac{\Hf_q(A+2x+2y+2z)\Hf_q(A+x+y+z)^2}{\Hf_q(A+2x+y+z)\Hf_q(A+x+2y+z)\Hf_q(A+x+y+2z)}\notag\\
  &\times \frac{\Hf_q(a+b+d+e+x+y+z)\Hf_q(a+c+d+f+x+y+z)\Hf_q(b+c+e+f+x+y+z)}{\Hf_q(a+d+e+f+x+y+z)^2\Hf_q(b+d+e+f+x+y+z)^2\Hf_q(c+d+e+f+x+y+z)^2}\notag\\
  &\times\frac{\Hf_q(a+d+x+y)\Hf_q(b+e+y+z)\Hf_q(c+f+z+x)}{\Hf_q(a+b+y)\Hf_q(b+c+z)\Hf_q(c+a+x)}\notag\\
  &\times \frac{\Hf_q(A-a+x+y+2z)\Hf_q(A-b+2x+y+z)\Hf_q(A-c+x+2y+z)}{\Hf_q(b+c+e+f+x+y+2z)\Hf_q(c+a+d+f+2x+y+z)\Hf_q(a+b+d+e+x+2y+z)},
\end{align}}
\normalsize where the sum on the left-hand side is taken over all GPPs $\pi$ fitting in the  compound box $\mathcal{B}\begin{pmatrix}x&y&z\\a&b&c\\d&e&f\end{pmatrix}$, and where $A=a+b+c+d+e+f$.
\end{thm}

\medskip

The goal of this paper is to prove Theorem \ref{qmain} by using Kuo's graphical condensation method, or simply \emph{Kuo condensation}. The method was introduced by Eric H. Kuo \cite{Kuo04} in his elegant proof of the well-known Aztec diamond theorem by Elkies, Kuperberg, Larsen and Propp \cite{Elkies, Elkies2}. Kuo condensation has become a powerful tool in enumeration of tilings. We refer the reader to e.g. \cite{YYZ,YZ,Kuo06,speyer,Ciucu3,Ful} for various aspects and generalizations of Kuo condensation; and e.g. \cite{CK2,CL,CF14,CF15,KW,Lai15a,Lai15b,Lai15c,Tri,LMNT,Zhang} for recent applications of the method.

\medskip

The rest of the paper is organized as follows. In Section 2, we quote several fundamental results, which will be employed in our proofs. Section 3 is devoted to the $q$-enumerations of lozenge tilings of two new families of dented hexagons. These $q$-enumerations are the key of the proof of Theorem \ref{qmain} in Section 4. We will show a consequence of Theorem \ref{main} in enumeration of (ordinary) plane partitions with certain constrains in Section 5. Finally, we conclude the paper by giving serval remarks.

\section{Preliminaries}

A \emph{perfect matching} of a graph $G$ is a collection of disjoint edges covering all vertices of $G$. The \emph{dual graph} of a region $R$ is the graph whose vertices are unit triangles in $R$ and whose edges connect precisely two unit triangles sharing an edge. Each edge of the dual graph carries the same weight as the corresponding lozenge in the region. The tilings of a region can be identified with the perfect matchings of its dual graph. The weight of a perfect matching is the product of the weights of its edges. The sum of the weights of  all perfect matchings in $G$ is called the \emph{matching generating function} of $G$, and denoted by $\M(G)$. We define similarly the \emph{tiling generating function} $\M(R)$ of the weighted region  $R$.

A \emph{forced lozenge} in a region $R$  is a lozenge contained in any tiling of $R$. Assume that we remove several forced lozenges $l_1,l_2\dotsc,l_n$ from the region $R$, and denote by $R'$  the resulting region. Then one readily obtains that
\begin{equation}\label{forcedeq}
\M(R)=\M(R')\prod_{i=1}^{n}wt(l_i),
\end{equation}
where $wt(l_i)$ is the weight of the lozenge $l_i$. The equality (\ref{forcedeq}) will be employed often in our proofs.

It is easy to see that if a region $R$ admits a tiling, then the numbers of up-pointing and down-pointing unit triangles in $R$ are equal. If a region $R$ satisfies the above balancing condition, we say that $R$ is \emph{balanced}. We have the following  generalization of (\ref{forcedeq}):

\begin{lem}[Region Splitting Lemma]\label{GS}
Let $R$ be a balanced region. Assume that a sub-region $Q$ of $R$ satisfies the following two conditions:
\begin{enumerate}
\item[(i)] \text{\rm{(Separating Condition)}} There is only one type of unit triangles (up-pointing or down-pointing unit triangles) running along each side of the border between $Q$ and $R-Q$.

\item[(ii)] \text{\rm{(Balancing Condition)}} $Q$ is balanced.
\end{enumerate}
Then
\begin{equation}\label{GSeq}
\M(R)=\M(Q)\, \M(R-Q).
\end{equation}
\end{lem}
\begin{proof}
Let $G$ be the dual graph of $R$, and $H$ the dual graph of $Q$. Then $H$ satisfies the conditions  in \cite[Lemma 3.6(a)]{Tri}, so $\M(G)=\M(H)\, \M(G-H)$.
Then (\ref{GSeq}) follows.
\end{proof}

The following Kuo's theorems are the key of our proofs.
\begin{thm}[Theorem 5.1 in \cite{Kuo04}]\label{kuothm1}
Let $G=(V_1,V_2,E)$ be a (weighted) bipartite planar graph in which $|V_1|=|V_2|$. Assume that  $u, v, w, s$ are four vertices appearing in a cyclic order on a face of $G$ so that $u,w \in V_1$ and $v,s \in V_2$. Then
\begin{equation}\label{kuoeq1}
\M(G)\M(G-\{u, v, w, s\})=\M(G-\{u, v\})\M(G-\{ w, s\})+\M(G-\{u, s\})\M(G-\{v, w\}).
\end{equation}
\end{thm}

\begin{thm}[Theorem 5.3 in \cite{Kuo04}]\label{kuothm2}
Let $G=(V_1,V_2,E)$ be a (weighted) bipartite planar graph in which $|V_1|=|V_2|+1$. Assume that  $u, v, w, s$ are four vertices appearing in a cyclic order on a face of $G$ so that $u,$ $v,$ $w \in V_1$ and $s \in V_2$. Then
\begin{equation}\label{kuoeq2}
\M(G-\{v\})\M(G-\{u, w, s\})=\M(G-\{u\})\M(G-\{v, w, s\})+\M(G-\{w\})\M(G-\{v, w,s\}).
\end{equation}
\end{thm}

The GPPs fitting in the compound box $\mathcal{B}=\mathcal{B}\begin{pmatrix}x&y&z\\a&b&c\\d&e&f\end{pmatrix}$ yields a natural $q$-weight assignment on the tilings of the region $F=F\begin{pmatrix}x&y&z\\a&b&c\\d&e&f\end{pmatrix}$ as follows. View each lozenge tiling $T$ of $F$ as a pile of unit cubes $\pi_T$ (i.e. a GPP). Each right lozenge is now the top of a column of unit cubes in $\pi_T$. Assign to each right lozenge in $T$ a weight $q^x$, where $x$ is the number of unit cubes in the corresponding column. In particular, the weight of $T$ is exactly $q^{|\pi_T|}$. We denote by $wt_0$ this weight assignment (see Figure \ref{3dentweight}(a)).

\begin{figure}\centering
\includegraphics[width=6cm]{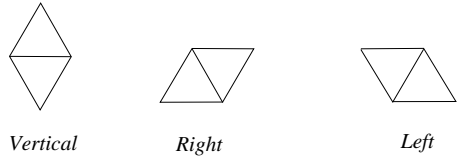}
\caption{Three orientations of lozenges.}
\label{rhumbustype}
\end{figure}

\begin{figure}\centering
\includegraphics[width=14cm]{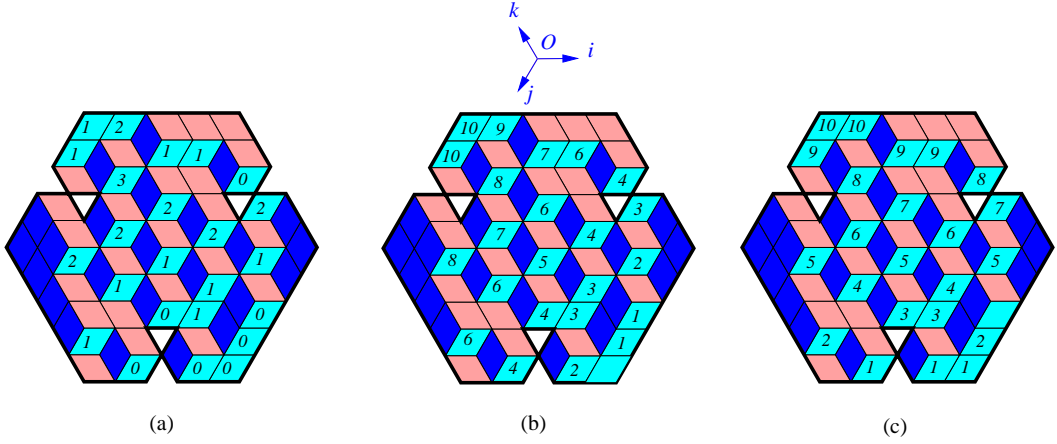}
\caption{Three $q$-weight assignments of a lozenges in a sample tiling of  a $F$-type region: (a) $wt_0$, (b) $wt_1$, (c) $wt_2$. The right lozenges with label $l$ have weight $q^l$.}
\label{3dentweight}
\end{figure}

Besides the natural $q$-weight assignment,  we consider the following two $q$-weight assignments. In both weight assignments, all left and vertical lozenges are weighted by $1$. The right lozenges are weighted as below:
\begin{enumerate}
\item[(1)] \emph{Assignment 1}. Each right lozenge is weighted by $q^{l}$, where $l$ is the distance between its left side and the southeast side of the region.
\item[(2)] \emph{Assignment 2}. Each right lozenge is weighted by $q^{k}$, where $k$ is the distance between its top and the base of the region.
\end{enumerate}
We denote by $wt_1$ and $wt_2$ the above weight assignments. Figure \ref{3dentweight} shows the three $q$-weight assignments for a sample tiling of the region $F\begin{pmatrix}1&1&1\\1&1&1\\1&1&1\end{pmatrix}$. Unlike $wt_0$, the weight assignments $wt_1$ and $wt_2$ are independent from the choice of the tiling $T$ of $F$. Thus, $wt_1$ and $wt_2$ can be applied to any hexagon with defects. 
Hereafter, we use the notation $wt_i(T)$ and $M_i(F)$ for the weight of the tiling $T$ and the tiling generating function of $F$ corresponding to the weight assignment $wt_i$, for $i=0,1,2$.

In the next part of this section, we will show that the two assignments $wt_1$ and $wt_2$ are the same up to a multiplicative factor.

We define two $9$-variable functions from $\mathbb{Z}_{\geq 0}^9$ to $\mathbb{Z}_{\geq0}$ as follows:
\small{\begin{align}
\textbf{g}&=\textbf{g}\begin{pmatrix}x&y&z\\a&b&c\\d&e&f\end{pmatrix}\notag\\
&:=(x+z+d+f)\binom{y+b+1}{2}+e\binom{b+1}{2}+a(c+x)(a+b+y)+a\binom{x+c+1}{2}\notag\\
&+(x+d)(a+b+y)(f+x+z)+(x+z+f)\binom{d+x+1}{2}+b(d+e+x+y)(a+b+y)\notag\\
&+b\binom{x+y+d+e+1}{2}+f(z+b)(a+b+d+e+x+2y+z)+(z+b)\binom{f+1}{2}\notag\\
&+xc(a+b+d+x+y)+x\binom{c+1}{2}
\end{align}}
\normalsize and
\small{\begin{align}
\textbf{h}&=\textbf{h}\begin{pmatrix}x&y&z\\a&b&c\\d&e&f\end{pmatrix}\notag\\
&:=b\binom{x+z+d+e+f+1}{2}+y\binom{x+z+d+f+1}{2}+(x+c)\binom{a+1}{2}+xc(a+d)+c\binom{x+1}{2}\notag\\
&+(x+d)\binom{x+z+f+1}{2}+(x+d)(x+z+f)(a+d)+f\binom{b+z+1}{2}+(x+y+d+e)\binom{b+1}{2}\notag\\
&+f(z+b)(x+y+z+a+d+e+f)+b(x+y+z+a+d+e+f)(x+y+d+e).
\end{align}}

\normalsize

\begin{prop}\label{ratioprop} For any non-negative integers $a,b,c,d,e,f,x,y,z$
\begin{equation}\label{ratioeq1}
\M_1\left(F\begin{pmatrix}x&y&z\\a&b&c\\d&e&f\end{pmatrix}\right)=q^{\textbf{g}}\sum_{\pi}q^{|\pi|}
\end{equation}
and
\begin{align}\label{ratioeq2}
\M_2\left(F\begin{pmatrix}x&y&z\\a&b&c\\d&e&f\end{pmatrix}\right)=q^{\textbf{h}}\sum_{\pi}q^{|\pi|},
\end{align}
where the sums are taken over all GPPs $\pi$ fitting in the compound box $\mathcal{B}:=\mathcal{B}\begin{pmatrix}x&y&z\\a&b&c\\d&e&f\end{pmatrix}$.
\end{prop}
The proof of Proposition \ref{ratioprop} follows the lines of the proof of Proposition 3.1 in \cite{Tri}. 

\begin{figure}\centering
\includegraphics[width=10cm]{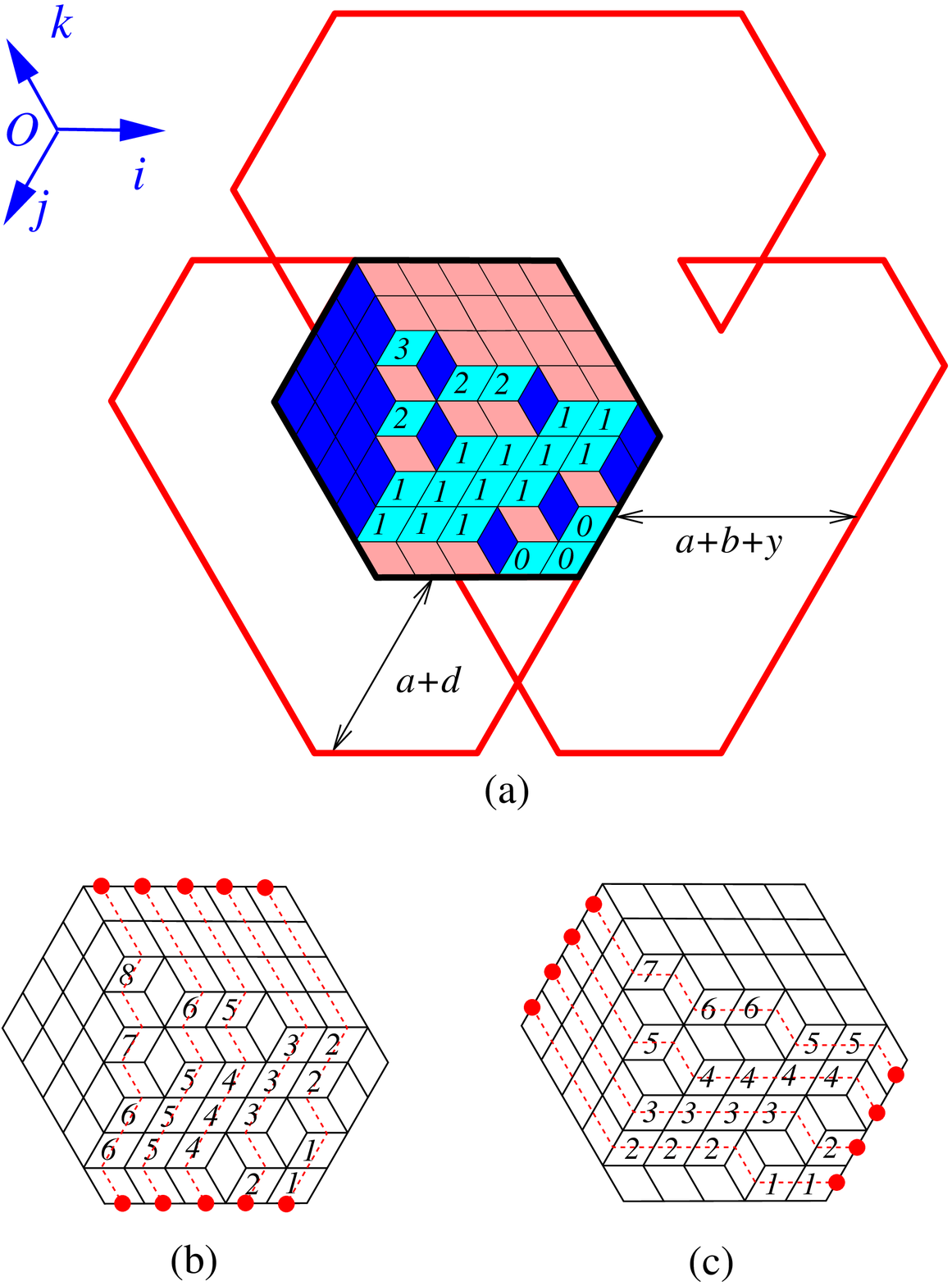}
\caption{(a) The sub-pile $\pi_4$ corresponding to the tiling in Figure \ref{3dentcombine}(a). (b) Encoding $T_4$ as a family of $b_4$ disjoint lozenge-paths. (c) Encoding $T_4$ as an $a_4$-tuple of disjoint lozenge-paths.}
\label{3dentratio}
\end{figure}

\begin{proof}
We assume that the component box $B_i$ ($1\leq i\leq 10$) in the compound box $\mathcal{B}$ has size $a_i\times b_i\times c_i$ (where $a_i,b_i,c_i$ can be determined explicitly from Figure \ref{3dentcombine}(d)). The base of $B_i$ is pictured by an $a_i\times b_i$ parallelogram $P_i$ as in Figure \ref{3dentcombine}(b). Assume that  the southeast side of $P_i$ is $x_i$ units to the left of the southeast side of $F$, and that the bottom of $P_i$ is $y_i$ units above the bottom of the region $F$ (we can get exactly the formulas of $x_i$ and $y_i$ from Figure \ref{3dentcombine}(b)).

Let $T$ be any lozenge tiling of $F$. View $T$ as a pile of unit cubes (i.e. a GPP) $\pi=\pi_{T}$. Divide $\pi$ into 10 disjoint sub-piles $\pi_i$'s fitting respectively in the component boxes $B_i$'s,  for $i=1,2,\dotsc,10$. Each sub-pile $\pi_i$ in turn yields a lozenge tiling $T_i$ of the hexagon $Hex(a_i,b_i,c_i)$ (see Figure \ref{3dentratio}(a) for an example).

Assume that the tiling $T$ is weighted by $wt_1$. This weight assignment yields a weight assignment for the lozenges in  the tiling $T_i$, for $i=1,2,\dotsc,10$. In particular,  each right lozenge in $T_i$ is weighted by $q^{x_i+l}$, where $l$ is the distance between the left side of the lozenge and the southeast side of the hexagon $Hex(a_i,b_i,c_i)$; and all left and vertical lozenges are weighted by $1$. Encode the tiling $T_i$ as a family of $b_i$ disjoint lozenge-paths connecting the top and the bottom of the hexagon (see Figure \ref{3dentratio}(b)). Dividing the weight of each right lozenge on the lozenge-path $j$ (from right to left) in $T_i$ by $q^{x_i+j}$, for $i=1,2,\dotsc,10$ and $j=1,2,\dotsc,b_i$,  we get back the weight assignment $wt_0$ on the tiling $T$. Since the end points of the lozenge-paths in each tiling $T_i$ are fixed, each of the above lozenge-path has exactly $a_i$ right lozenges.
Thus, the weight changing gives
\begin{equation}\label{ratioeq3}
\frac{\wt_1(T)}{\wt_0(T)}=\frac{\wt_1(T)}{q^{|\pi|}}=q^{\sum_{i=1}^{10}a_ib_ix_i+a_ib_i(b_i+1)/2}.
\end{equation}

Next, we assume that the tiling $T$ is weighted by $wt_2$. We now encode each tiling $T_i$ of $Hex(a_i,b_i,c_i)$ as an $a_i$-tuple of disjoint lozenge-paths connecting the northwest and the southeast sides of the hexagon (illustrated in Figure \ref{3dentratio}(b)). By dividing the weight of each right lozenge on the lozenge-path $j$ (from bottom to top) of  each tiling $T_i$ by $q^{y_i+j}$, for $i=1,2,\dotsc,10$ and $j=1,2,\dotsc,a_i$, we also get back the weight assignment $wt_0$ of $T$. Similar to the case of $wt_1$, we obtain
\begin{equation}\label{ratioeq4}
\frac{\wt_2(T)}{wt_0(T)}=\frac{\wt_2(T)}{q^{|\pi|}}=q^{\sum_{i=1}^{10}a_ib_iy_i+b_ia_i(a_i+1)/2}.
\end{equation}
 From Figures \ref{3dentcombine}(b) and (d), one gets the formulas for $a_i,b_i,x_i,y_i$ in terms of the 9 parameters $a,$ $b,$ $c,$ $d,$ $e,$ $f,$ $x,$ $y,$ $z$. Plugging these formulas into  (\ref{ratioeq3}) and (\ref{ratioeq4}), we get respectively (\ref{ratioeq1}) and (\ref{ratioeq2}).
\end{proof}

Apply the arguments in the above proof to the case when the compound box consists of a single component box, and using MacMahon's theorem, we get the following corollary.
\begin{cor}\label{lem0}
For any non-negative integers $a,b,c$
\begin{equation} \label{hexeq1}
\M_1\big(Hex(a,b,c)\big)=q^{ab(b+1)/2}\frac{\Hf_q(a)\Hf_q(b)\Hf_q(c)\Hf_q(a+b+c)}{\Hf_q(a+b)\Hf_q(b+c)\Hf_q(c+a)}
\end{equation}
and
\begin{equation}\label{hexeq2}
\M_2\big(Hex(a,b,c)\big)=q^{ba(a+1)/2}\frac{\Hf_q(a)\Hf_q(b)\Hf_q(c)\Hf_q(a+b+c)}{\Hf_q(a+b)\Hf_q(b+c)\Hf_q(c+a)}.
\end{equation}
\end{cor}
This corollary was also introduced as Corollary 3.2 in \cite{Tri}.

 \begin{figure}\centering
\includegraphics[width=5cm]{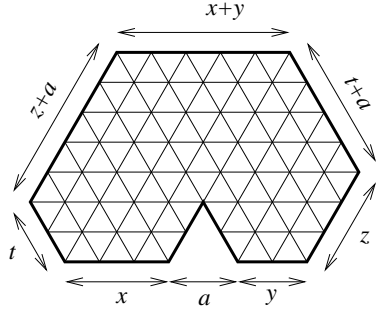}
\caption{The region $K_{2}(3,2,3,2)$.}
\label{holeyhex}
\end{figure}

\medskip

We remove an equilateral triangle of side-length $a$ from the south side of a hexagon with side-lengths $z+a$, $x+y,$  $t+a,$ $z,$ $x+y+a,$ $t$ as in Figure \ref{holeyhex}. Denote by $K_{a}(x,y,z,t)$ the resulting region. The following result was proven in \cite[Lemma 3.3]{Tri2}, based on a well-known bijection between the lozenge tilings of a $K$-type region and \emph{column-strict plane partitions} (see e.g. \cite{CLP} and \cite{car}), and the explicit formula of the generating function of the column-strict plane partitions (see e.g. \cite[pp. 374--375]{Stanley}).

\begin{lem}\label{qlem1}
For any non-negative integers $a,x,y,z,t$
\begin{align}
\M_2(K_{a}(x,y,z,t))&=q^{y\binom{z+1}{2}+x\binom{z+a+1}{2}}\frac{\Hf_q(a)\Hf_q(x)\Hf_q(y)\Hf_q(z)\Hf_q(t)}{\Hf_q(a+x)\Hf_q(a+y)\Hf_q(y+z)\Hf_q(t+x)}\notag\\
&\times\frac{\Hf_q(a+x+y)\Hf_q(a+y+z)\Hf_q(a+t+x)\Hf_q(a+x+y+z+t)}{\Hf_q(a+x+y+z)\Hf_q(a+x+y+t)\Hf_q(a+z+t)}.
\end{align}
\end{lem}

\section{Two new families of hexagons with dents}

In this section we $q$-enumerate lozenge tilings of two new hexagons with dents. We need these $q$-enumerations for the proof of Theorem \ref{qmain} in the next section.

\begin{figure}\centering
\includegraphics[width=6cm]{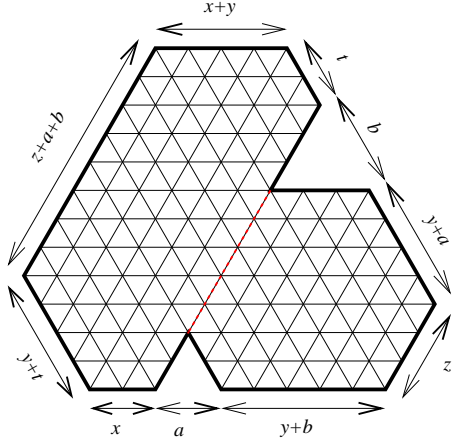}
\caption{The hexagon with two triangular dents $Q_{2,3}(2,2,3,2)$.}
\label{twodents}
\end{figure}

Starting with a hexagon of side-lengths $z+a+b$, $x+y$, $y+t+a+b,$ $z$, $x+y+a+b,$ $y+t$, we remove an $a$-triangle along the south side and a $b$-triangle along the northeast side of the hexagon as in Figure \ref{twodents}. Denote by $Q_{a,b}(x,y,z,t)$ the resulting region. We note that the left sides of the two triangular dents in our region are always on the same lattice line (illustrated by the dotted line in Figure \ref{twodents}).

Similar to the case of $F$-type regions, the lozenge tilings of $Q_{a,b}(x,y,z,t)$ are in bijection with the piles\footnote{Hereafter, we always assume that all piles of unit cubes have the same monotonicity as the GPPs.}
 of unit cubes fitting in the compound box $\mathcal{C}:=\mathcal{C}_{a,b}(x,y,z,t)$ (see Figure \ref{tiling2dent}(a)). To precise, the box $\mathcal{C}$ consists of $4$ component (rectangular) boxes $C_1,C_2,C_3,C_4$, whose bases are labelled by $1,2,3,4$ as in Figure \ref{tiling2dent}(b). Figure \ref{tiling2dent}(c) presents the projection of the box $\mathcal{C}$ on the $\textbf{Oij}$ plane. Finally, the connectivity of the component boxes in $\mathcal{C}$ is shown in \ref{tiling2dent}(d). Here we still use the notation $|\pi|$ for the volume of a pile of unit cubes $\pi$.
\begin{figure}\centering
\includegraphics[width=12cm]{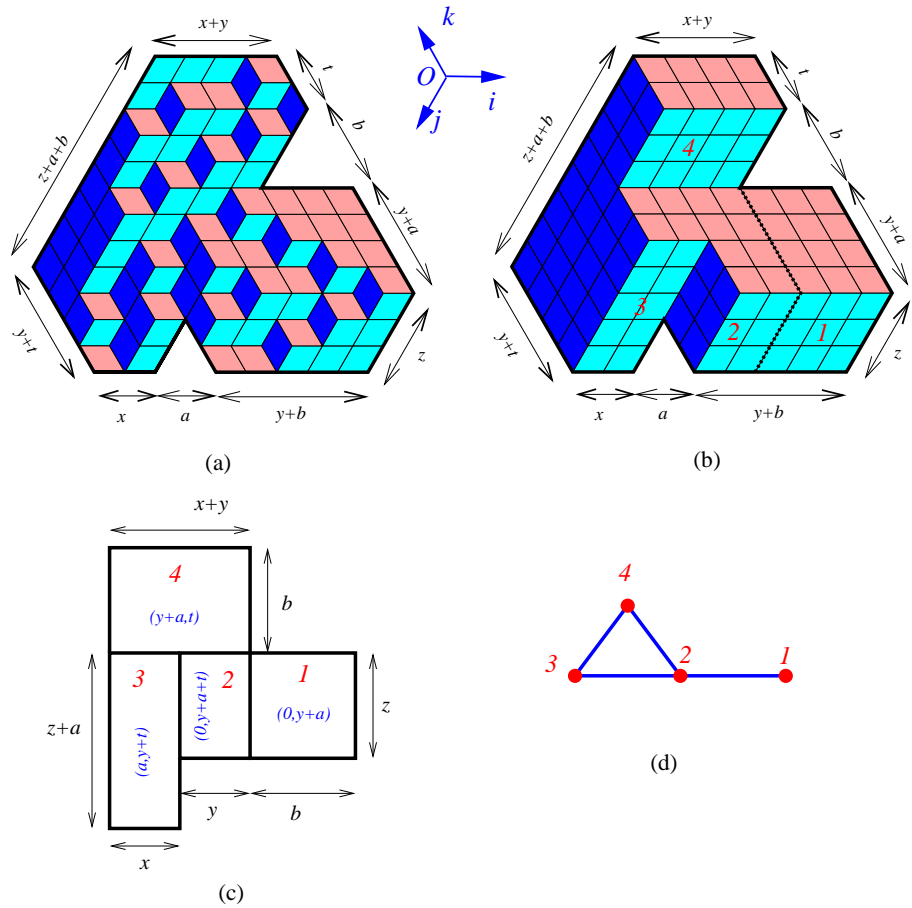}
\caption{(a) Viewing a lozenge tiling of the region $Q_{a,b}(x,y,z,t)$ as a pile of unit cubes fitting in the compound box $\mathcal{C}$. (b) The 3-D picture of the compound box $\mathcal{C}$. (c) The projection of the box $\mathcal{C}$ on the $\textbf{Oij}$ plane. (d) The connectivity of the component boxes in $\mathcal{C}$.}
\label{tiling2dent}
\end{figure}

\begin{thm}\label{qpremain1} For non-negative integers $a,$ $b,$ $x,$ $y,$ $z$, $t$
\begin{align}\label{qtwodenteq1}
q^{E}\sum_{\pi}q^{|\pi|}&=\M_2\left(Q_{a,b}(x,y,z,t)\right)\notag\\
&=q^{E}\frac{\Hf_q(x)\Hf_q(y)\Hf_q(z)\Hf_q(t)\Hf_q(a)\Hf_q(b)}{\Hf_q(a+x)\Hf_q(b+t)\Hf_q(a+b+y)}\notag\\
  &\times\frac{\Hf_q(a+b+x+2y+z+t)\Hf_q(a+b+x+y+t)}{\Hf_q(a+b+y+z+t)\Hf_q(a+b+x+2y+t)\Hf_q(a+b+x+y+z)}\notag\\
  &\times \frac{\Hf_q(a+x+y)\Hf_q(b+y+t)\Hf_q(a+b+y+t)^2}{\Hf_q(x+y+t)\Hf_q(a+y+z)\Hf_q(b+y+z)},
\end{align}
\normalsize where
\[E=(y+b)\binom{z+1}{2}+x\binom{a+z+1}{2}+b(x+y)(a+y+z)+(x+y)\binom{b+1}{2}\]
and where the sum is taken over all piles $\pi$ fitting in the compound box $\mathcal{C}$.
\end{thm}

\begin{figure}\centering
\includegraphics[width=12cm]{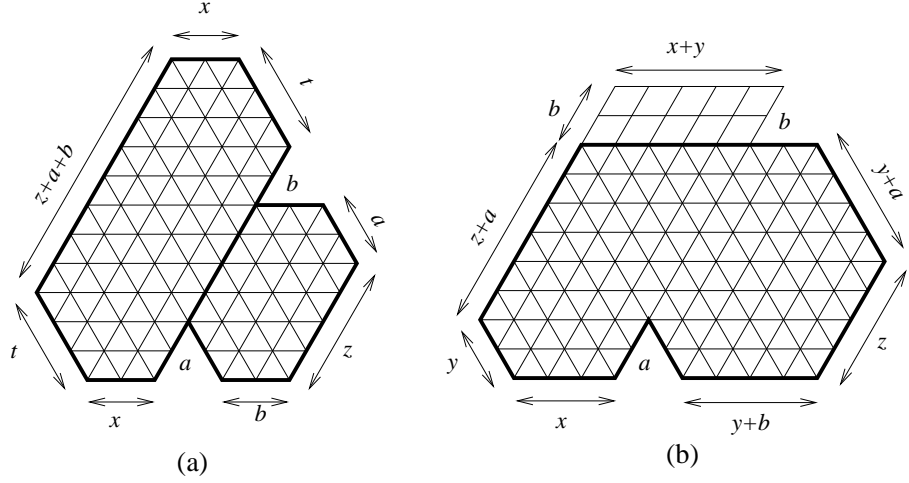}
\caption{The base cases when (a) $y=0$ and (b) $t=0$ in the proof of Theorem \ref{qpremain1}.}
\label{BC2dent}
\end{figure}

\begin{figure}\centering
\includegraphics[width=6cm]{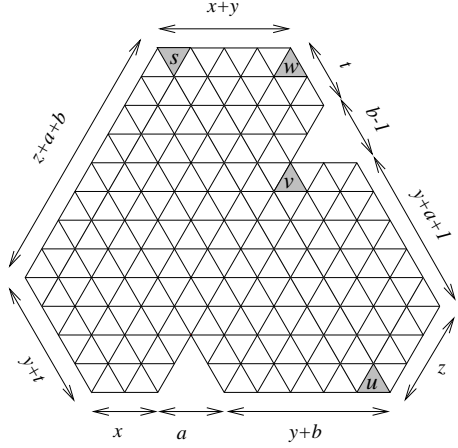}
\caption{How we apply Kuo condensation to a $Q$-type region.}
\label{Kuo2dent}
\end{figure}

\begin{figure}\centering
\includegraphics[width=13cm]{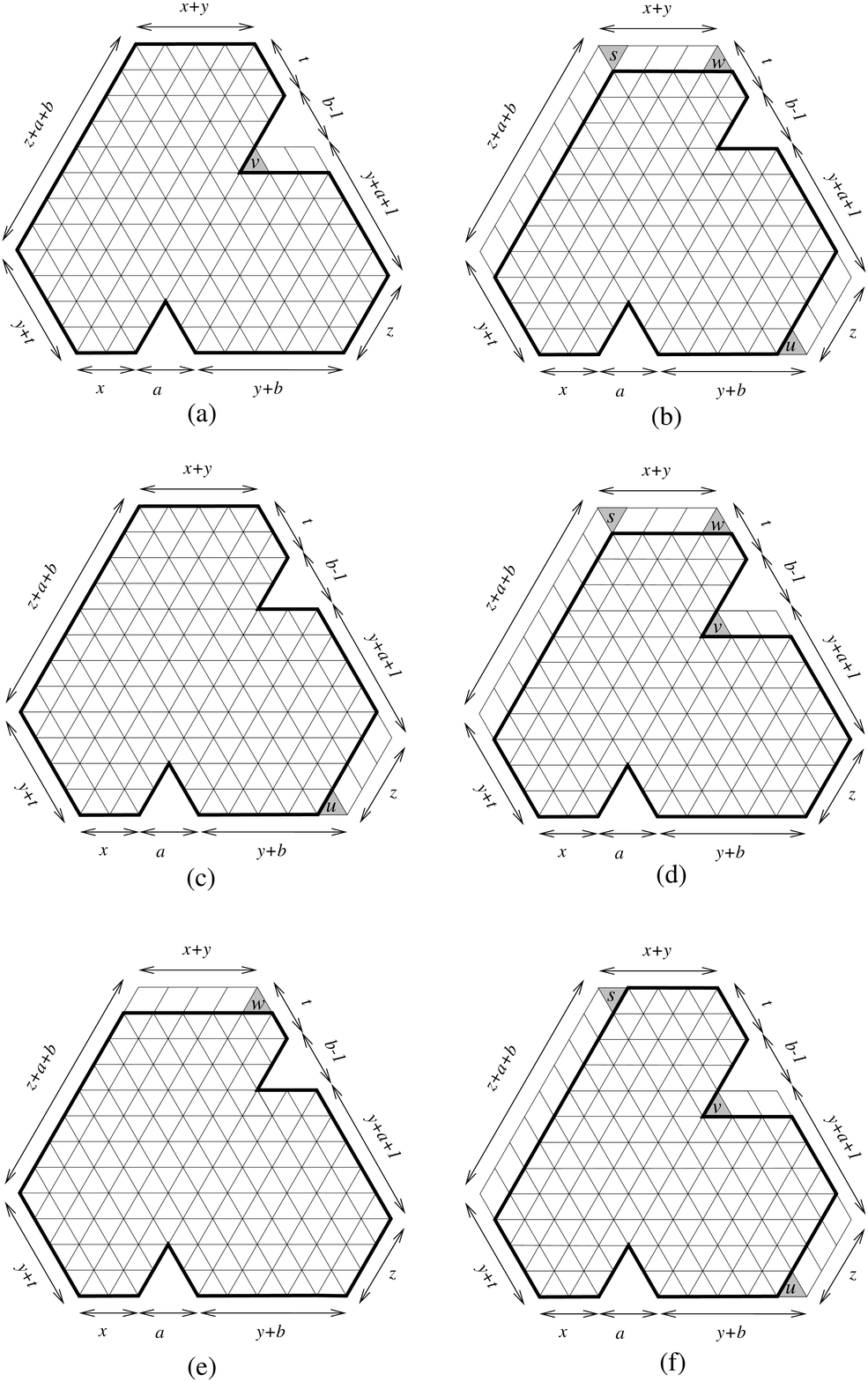}
\caption{Obtaining recurrence of the tiling generating functions of $Q$-type regions.}
\label{Kuo2dentb}
\end{figure}

\begin{proof}
First, by the same arguments as in the proof of Proposition \ref{ratioprop}, we have for any tiling $T$ of the region $Q:=Q_{a,b}(x,y,z,t)$:
\begin{equation}\label{ratio2dent2}
\wt_2(T)/q^{|\pi_T|}=q^{\sum_{i=1}^{4}a_ib_iy_i+b_ia_i(a_i+1)/2},
\end{equation}
where the component box $C_i$ of $\mathcal{C}$ has size $a_i\times b_i \times c_i$, and the base $P_i$ of $C_i$ is $y_i$ units above the base of the region in Figure \ref{tiling2dent}(b). Obtaining the formulas for $a_i,b_i,y_i$ in terms of $a,b,x,y,z,t$ from Figure \ref{tiling2dent}, and summing the equation (\ref{ratio2dent2}) over all tilings $T$ of $Q$, we get
\begin{equation}
\M_2(Q)/\sum_{\pi}q^{|\pi|}=q^{E}.
\end{equation}
This means that we only need to prove the second equal sign in (\ref{qtwodenteq1}).

\medskip

Next, we prove the second equal sign in (\ref{qtwodenteq1}) by induction on $y+t+2b$. Our base cases are the situations when at least one of the three parameters $b,y,t$ is equal to $0$.

Assume that our region $Q$ is weighted by $\wt_2$. If $b=0$, our region is exactly a $K$-type region in Lemma \ref{qlem1}, and the second equal sign in (\ref{qtwodenteq1}) follows.
If $y=0$, Region-splitting Lemma \ref{GS} gives us
\begin{equation}
\M_2\big(Q_{a,b,}(x,0,z,t)\big)=\M_2\big(Hex(z+a+b,x,t)\big)\M_2\big(Hex(z,b,a)\big)
\end{equation}
(see Figure \ref{BC2dent}(a)). Then the second equal sign of (\ref{qtwodenteq1}) follows from Corollary \ref{lem0}.
 If $t=0$, there are several forced right lozenges on the top of our region as in Figure \ref{BC2dent}(b). By removing these forced lozenges, we obtain a $K$-type region weighted by $wt_2$. Collecting the weights of those forced lozenges, we get
\begin{equation}
\M_2\big(Q_{a,b,}(x,y,z,0)\big)=q^{b(x+y)(y+z+a)+(x+y)\binom{b+1}{2}}\M_2\big(K_{a}(x,y+b,z,y)\big),
\end{equation}
and the second equal sign in (\ref{qtwodenteq1}) follows from Lemma \ref{qlem1} again.

\medskip

For the induction step, we assume that  $b,t,y>0$, and that (\ref{qtwodenteq1}) holds for any $Q$-type regions in which the sum of the $y$-parameter, the $t$-parameter, and twice the $b$-parameter is strictly less than $y+t+2b$.

We apply Kuo's Theorem \ref{kuothm2} to the dual graph $G$ of the region $R$ weighted by $\wt_2$, where $R$ is the region obtained from the region $Q_{a,b}(x,y,z,t)$ by adding a band of $2b-1$ unit triangles along the bottom of the $b$-hole (see Figure \ref{Kuo2dent}). The unit triangles corresponding to the four vertices $u,v,w,s$ are illustrated as the four shaded unit triangles in Figure \ref{Kuo2dent}. In particular, the bottommost shaded unit triangle corresponds to $u$, and the unit triangles corresponding to $v,w,s$ are the next shaded unit triangles when we move counter-clockwise from the bottommost one. By collecting the weights of the lozenges forced by the removal of the shaded unit triangles as in Figures \ref{Kuo2dentb}(a)--(f), we have respectively
\begin{equation}\label{q2denteq1}
\M(G-\{v\})=\M_2\big(Q_{a,b}(x,y,z,t)\big),
\end{equation}
\begin{equation}\label{q2denteq2}
\M(G-\{u,w, s\})=q^{(x+y-1)(y+z+t+a+b)}\M_2\big(Q_{a,b-1}(x,y,z+1,t-1)\big),
\end{equation}
\begin{equation}\label{q2denteq3}
\M(G-\{u\})=\M_2\big(Q_{a,b-1}(x,y,z+1,t)\big),
\end{equation}
\begin{equation}\label{q2denteq4}
\M(G-\{v,w,s\})=q^{(x+y-1)(y+z+t+a+b)}\M_2\big(Q_{a,b}(x,y,z,t-1)\big),
\end{equation}
\begin{equation}\label{q2denteq5}
\M(G-\{w\})=q^{(x+y)(y+z+t+a+b)}\M_2\big(Q_{a,b-1}(x,y+1,z,t-1)\big),
\end{equation}
and
\begin{equation}\label{q2denteq6}
\M(G-\{u,v,s\})=\M_2\big(Q_{a,b}(x,y-1,z+1,t)\big).
\end{equation}

Substituting the above six identities (\ref{q2denteq1}) -- (\ref{q2denteq6}) into the equation (\ref{kuoeq2}) in Kuo's Theorem \ref{kuothm2}, we get
\begin{align}\label{q2denteq7}
\M_2\big(Q_{a,b}(x,y,z,t)\big)&\M_2\big(Q_{a,b-1}(x,y,z+1,t-1)\big)=\notag\\
&\M_2\big(Q_{a,b-1}(x,y,z+1,t)\big)\M_2\big(Q_{a,b}(x,y,z,t-1)\big)\notag\\
&+q^{y+z+t+a+b}\M_2\big(Q_{a,b-1}(x,y+1,z,t-1)\big)\M_2\big(Q_{a,b}(x,y-1,z+1,t)\big).
\end{align}
By the induction hypothesis, all regions in (\ref{q2denteq7}), except for the first one, have the tiling generating function given by (\ref{qtwodenteq1}). Substituting these formulas into (\ref{q2denteq7}) and simplifying, one gets that $\M_2\big(Q_{a,b}(x,y,z,t)\big)$ is given exactly by the expression after the second equal sign of (\ref{q2denteq7}). We finish our proof here.
\end{proof}

To prove Theorem \ref{qmain}, we need also the following variation of Theorem \ref{qpremain1}.

\begin{thm}\label{qpremain2}
Assume that all right and vertical lozenges of the region $Q_{a,b}(x,y,z,t)$ have weight 1. We now assign to any left lozenge a weight $q^{l}$, where $l$ is the distance from the top of the lozenge to the bottom of the region. Denote by $\M_3(Q_{a,b}(x,y,z,t))$ the tiling generating function of $Q_{a,b}(x,y,z,t)$ with respect to the new weight assignment. Then
\begin{align}\label{qtwodenteq2}
\M_3\big(Q_{a,b}(x,y,z,t)\big)&=q^{b\binom{a+y+1}{2}+y\binom{a+y+t+1}{2}+x\binom{y+t+1}{2}}\frac{\Hf_q(x)\Hf_q(y)\Hf_q(z)\Hf_q(t)\Hf_q(a)\Hf_q(b)}{\Hf_q(a+x)\Hf_q(b+t)\Hf_q(a+b+y)}\notag\\
  &\times\frac{\Hf_q(a+b+x+2y+z+t)\Hf_q(a+b+x+y+t)}{\Hf_q(a+b+y+z+t)\Hf_q(a+b+x+2y+t)\Hf_q(a+b+x+y+z)}\notag\\
  &\times \frac{\Hf_q(a+x+y)\Hf_q(b+y+t)\Hf_q(a+b+y+z)^2}{\Hf_q(x+y+t)\Hf_q(a+y+z)\Hf_q(b+y+z)}.
\end{align}
\end{thm}
We give here two different proofs of Theorem \ref{qpremain2}. The first proof follows the lines of the proof of Theorem \ref{qpremain1}, while the second proof is based on the arguments in the proof of Proposition \ref{ratioprop}.
\begin{proof} [The first proof of Theorem \ref{qpremain2}]
\normalsize Similar to Theorem \ref{qpremain1}, we prove (\ref{qtwodenteq2}) by induction on $y+t+2b$ with the base cases are: $b=0$, $y=0$, and $t=0$.

\medskip

Consider the region $Q_{a,b}(x,y,z,t)$ weighted by the new weight assignment.

 If $b=0$, we reflect our region about a vertical line to get the region $K_{a}(y,x,y+t,z)$ weighted by $\wt_2$. Then (\ref{qtwodenteq2}) follows from Lemma \ref{qlem1}. If $y=0$, we also split the region into two parts by using Region-splitting Lemma \ref{GS} as in Figure \ref{BC2dent}(a).  However, we need to reflect these parts about a vertical line to get two hexagons weighted by $\wt_2$. In particular, we get
 \begin{equation}
 \M_3\big(Q_{a,b}(x,0,z,t)\big)=\M_2\big(Hex(t,x,z+a+b)\big)\M_2\big(Hex(a,b,z)\big),
 \end{equation}
 then (\ref{qtwodenteq2}) follows from Corollary \ref{lem0}. Finally, if $t=0$, by removing the forced right lozenges (which have all weight 1 in the new weight assignment) as in Figure \ref{BC2dent}(b), we get a $K$-type region. However, we also need to reflect this region about a vertical line to get back the weight assignment $\wt_2$. To precise, we get
  \begin{equation}
 \M_3\big(Q_{a,b}(x,y,z,0)\big)=\M_2\big(K_{a}(y+b,x,y,z)\big),
 \end{equation}
 and  (\ref{qtwodenteq2}) is implied by Lemma \ref{qlem1} again.

\medskip

 The induction step is completely analogous to the induction step in the proof of Theorem \ref{qpremain1}, the only difference is that the forced lozenges have \emph{different} weights. In particular, by collecting the new weights of the forced lozenges in Figure \ref{Kuo2dentb}, we get the following new recurrence for $\M_3$-generating functions:
\begin{align}\label{q2denteqn}
\M_3\big(Q_{a,b}(x,y,z,t)\big)&\M_3\big(Q_{a,b-1}(x,y,z+1,t-1)\big)=\notag\\
&\M_3\big(Q_{a,b-1}(x,y,z+1,t)\big)\M_3\big(Q_{a,b}(x,y,z,t-1)\big)\notag\\
&+\M_3\big(Q_{a,b-1}(x,y+1,z,t-1)\big)\M_3\big(Q_{a,b}(x,y-1,z+1,t)\big),
\end{align}
and by the induction hypothesis and some simplifications, one gets that $\M_3\big(Q_{a,b}(x,y,z,t)\big)$ is given exactly by the expression on the right-hand side of (\ref{qtwodenteq2}).
\end{proof}

\begin{proof}[The second proof of Theorem \ref{qpremain2}]
Let $T$ be any lozenge tiling the region $Q:=Q_{a,b}(x,y,z,t)$. We introduce an analog $wt'$ of the natural  $q$-weight assignment $wt_0$ as follows. View $T$ as a pile $\pi=\pi_T$ of unit cubes fitting in the compound box $\mathcal{C}$. Each left lozenge is now the front face of a column of unit cubes parallel to $\overrightarrow{\textbf{Oj}}$. To distinguish to the normal columns, we call these columns \emph{$j$-columns}. We assign to each left lozenge a weight $\left(\frac{1}{q}\right)^{x}$, where $x$ is the number of unit cubes in the corresponding $j$-column (see Figure \ref{tiling2dentb}(a); the left lozenge having label $x$ is weighted by $\left(\frac{1}{q}\right)^{x}$). Finally, all right and vertical lozenges are weighted by $1$. It is easy to see that
\[wt'(T)=\left(\frac{1}{q}\right)^{|\pi|}=\frac{1}{wt_0(T)}.\]

\begin{figure}\centering
\includegraphics[width=12cm]{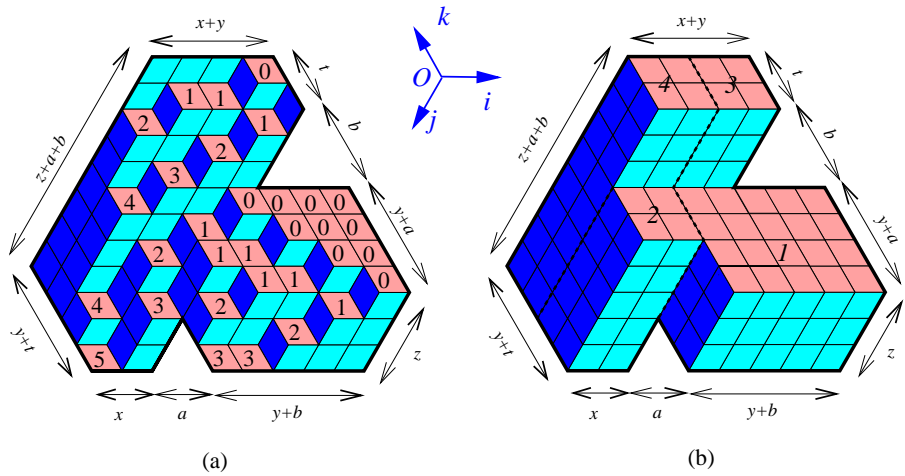}
\caption{(a) The weight assignment $wt'$ on the lozenges of $Q_{a,b}(x,y,z,t)$. (b) New decomposition of the compound box $\mathcal{C}$ into 4 component boxes with bases labelled by $1,2,3,4$.}
\label{tiling2dentb}
\end{figure}

Next, we show that the new weight assignment and $wt'$ are only different by a multiplicative factor. Partition $\mathcal{C}$ into four new component boxes $C'_1,C'_2,C'_3,C'_4$ with the far faces are labelled by $1,2,3,4$ as in Figure \ref{tiling2dentb}(b). The base of the component box $C'_i$ is pictured as a parallelogram $P'_i$ staying $x_i$ units above the base of the region $Q$. This yields a partition of the pile $\pi$ into four sub-piles $\pi_i$'s fitting in $C'_i$ of size $a_i\times b_i \times c_i$. Each sub-pile $\pi_i$ in turn gives a tiling $T_i$ of the hexagon $Hex(a_i,b_i,c_i)$. Encode each tiling $T_i$ as a $c_i$-tuple of disjoint lozenge-paths connecting the northeast and the southwest sides of the hexagon. Divide the weight of each left lozenge in path $j$ (from right to left) of each tiling $T_i$ by $q^{x_i+j}$, we get back the weight assignment $wt'$ and
\begin{equation}
\frac{wt_3(T)}{wt'(T)}=q^{(x+y)\left(\binom{a+b+y+z+t+1}{2}-\binom{a+b+y+z+1}{2}\right)+x\left(\binom{a+y+z+1}{2}-\binom{a+z+1}{2}\right)+(y+b)\left(\binom{a+y+z+1}{2}-\binom{z+1}{2}\right)}.
\end{equation}
Summing up over all tilings $T$ of $Q$, we get
\begin{equation}
\frac{\M_3(Q)}{\sum_{\pi}{q}^{-|\pi|}}=q^{(x+y)\left(\binom{a+b+y+z+t+1}{2}-\binom{a+b+y+z+1}{2}\right)+x\left(\binom{a+y+z+1}{2}-\binom{a+z+1}{2}\right)
+(y+b)\left(\binom{a+y+z+1}{2}-\binom{z+1}{2}\right)},
\end{equation}
and the theorem follows from Theorem \ref{qpremain1} (\emph{with $q$ is repalced by $q^{-1}$}) and the simple fact $[n]_{q^{-1}}=[n]_q/q^{n-1}$.
\end{proof}

\begin{figure}\centering
\includegraphics[width=7cm]{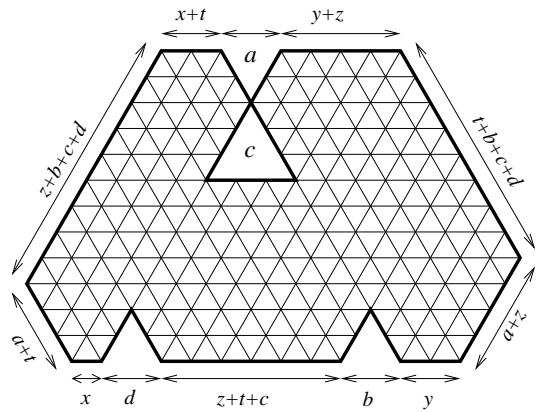}
\caption{A $B$-type region.}
\label{dentbar}
\end{figure}

 Next, we consider a new family of hexagons with three dents.

 Let $a,b,c,d,x,y,z,t$ be eight non-negative integers.  We remove a bowtie from the north side, and two triangles  from the south side of a hexagon of side-lengths $z+b+c+d$, $x+y+z+t+a$, $t+b+c+d$, $z+a$, $x+y+z+t+b+c+d$, $t+a$. The locations and the sizes of the dents are indicated as in Figure \ref{dentbar} (for the case $a=2$, $b=2$, $c=3$, $d=2$, $x=1$, $y=2$, $z=2$, $t=1$). Denote by $B\begin{pmatrix}x&y&z&t\\a&b&c&d\end{pmatrix}$ the resulting region.

The lozenge tilings of the region $B\begin{pmatrix}x&y&z&t\\a&b&c&d\end{pmatrix}$ can be viewed as piles of unit cubes fitting in the $7$-component compound box $\mathcal{D}:=\mathcal{D}\begin{pmatrix}x&y&z&t\\a&b&c&d\end{pmatrix}$ illustrated in Figure \ref{tilingdentbar}. In particular, $\mathcal{D}$ is decomposed into $7$ component boxes $D_1,D_2,\dotsc,D_7$ with bases labelled by $1,2,\dotsc,7$ in Figure \ref{tilingdentbar}(b).

\begin{figure}\centering
\includegraphics[width=15cm]{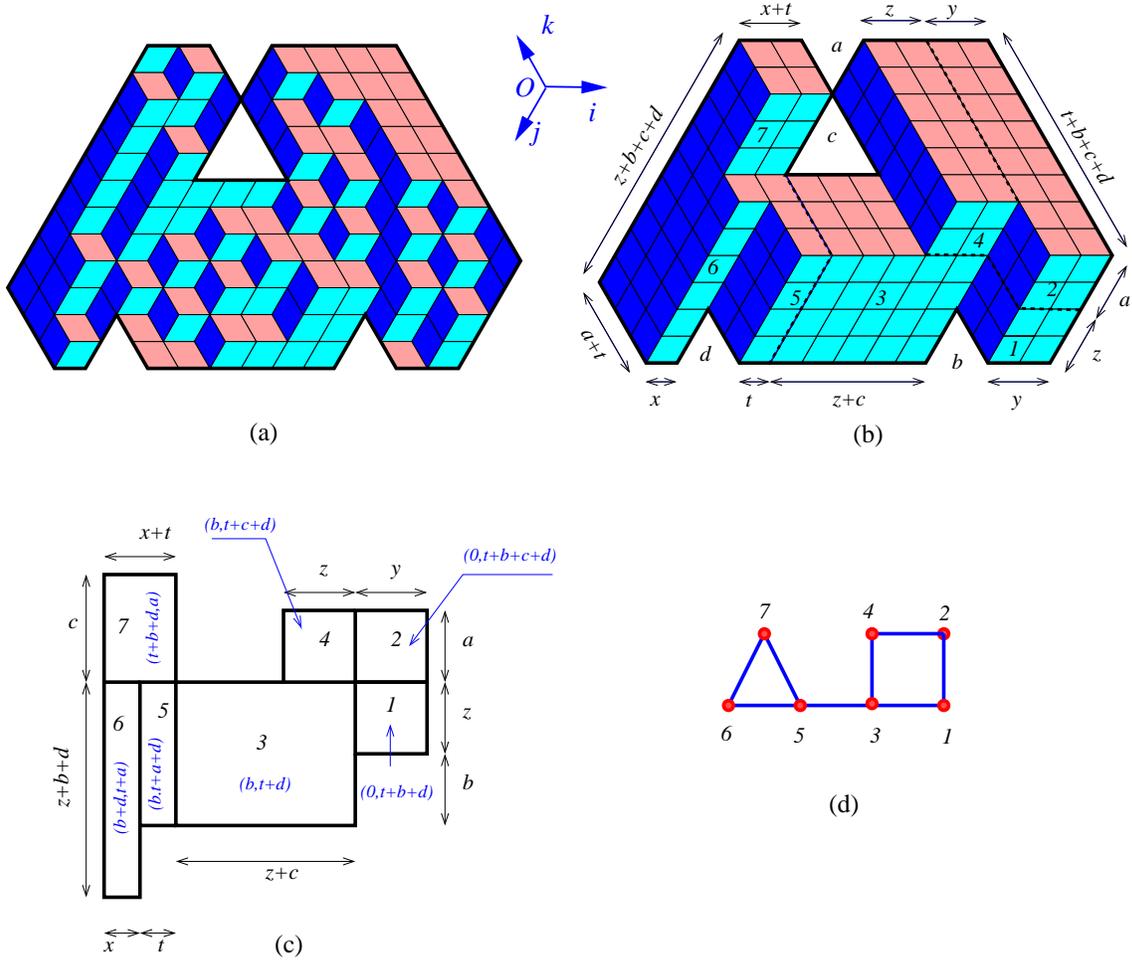}
\caption{ (a) Viewing a lozenge tiling of a $B$-type region as a pile of unit cubes fitting in the compound box $\mathcal{D}$. (b) A 3-D picture of the compound box $\mathcal{D}$ (or the empty pile). (c) The projection of $\mathcal{D}$ on the $\textbf{Oij}$ plane. (d) The  connectivity of the component boxes in $\mathcal{D}$.}
\label{tilingdentbar}
\end{figure}

\begin{thm}\label{qdentbarthm} For non-negative integers $a,$ $b,$ $c,$ $d,$ $x,$ $y,$ $z$, $t$
\begin{align}\label{qdentbareq}
q^{A}&\sum_{\pi}q^{|\pi|}=\M_2\left(B\begin{pmatrix}x&y&z&t\\a&b&c&d\end{pmatrix}\right)\notag\\
&=q^{A}\frac{\Hf_q(x)\Hf_q(y)\Hf_q(z)\Hf_q(t)\Hf_q(a)^2\Hf_q(b)\Hf_q(c)\Hf_q(d)}{\Hf_q(a+c)\Hf_q(b+x)\Hf_q(d+y)}\notag\\
&\times\frac{\Hf_q(a+b+c+d+y+z+2t)\Hf_q(a+b+c+d+x+2z+t)}{\Hf_q(a+c+d+y+z+2t)\Hf_q(a+b+c+x+2z+t)}\notag\\
 &\times \frac{\Hf_q(a+b+c+d+x+y+2z+2t)\Hf_q(a+b+c+d+x+y+z+t)}{\Hf_q(a+b+c+d+x+y+z+2t)\Hf_q(a+b+c+d+x+y+2z+t)}\notag\\
  &\times\frac{\Hf_q(a+b+c+x+z+t)\Hf_q(a+c+d+y+z+t)\Hf_q(b+c+d+z+t)^3}{\Hf_q(b+c+d+x+z+t)\Hf_q(b+c+d+y+z+t)\Hf_q(a+b+c+d+z+t)^2}\notag\\
  &\times \frac{\Hf_q(d+y+t)\Hf_q(b+x+z)\Hf_q(a+c+z+t)}{\Hf_q(b+c+z+t)\Hf_q(c+d+z+t)\Hf_q(a+x+z)\Hf_q(a+y+t)\Hf_q(b+d+z+t)},
\end{align}
\normalsize where
\begin{align}
A=A\begin{pmatrix}x&y&z&t\\a&b&c&d\end{pmatrix}&:=y\binom{a+z+1}{2}+(c+z+t)\binom{z+b+1}{2}+x\binom{b+d+z+1}{2}\notag\\
&+az(b+z)+z\binom{a+1}{2}+c(x+t)(b+d+z+t)+(x+t)\binom{c+1}{2}
\end{align}
and where the sum is taken over all piles $\pi$ fitting in the compound box $\mathcal{D}$.
\end{thm}

\normalsize

\begin{figure}\centering
\includegraphics[width=14cm]{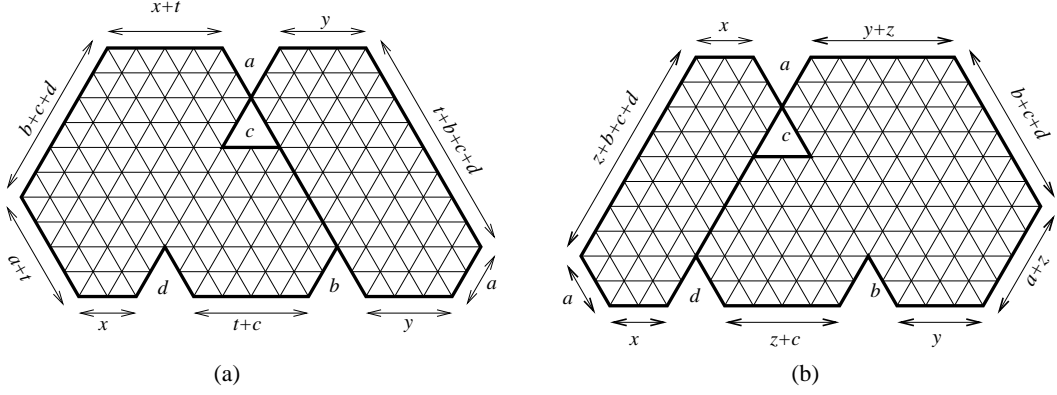}
\caption{Base cases of Theorem \ref{qdentbarthm} when (a) $z=0$ and (b) $t=0$.}
\label{BCdentbar}
\end{figure}

\begin{figure}\centering
\includegraphics[width=14cm]{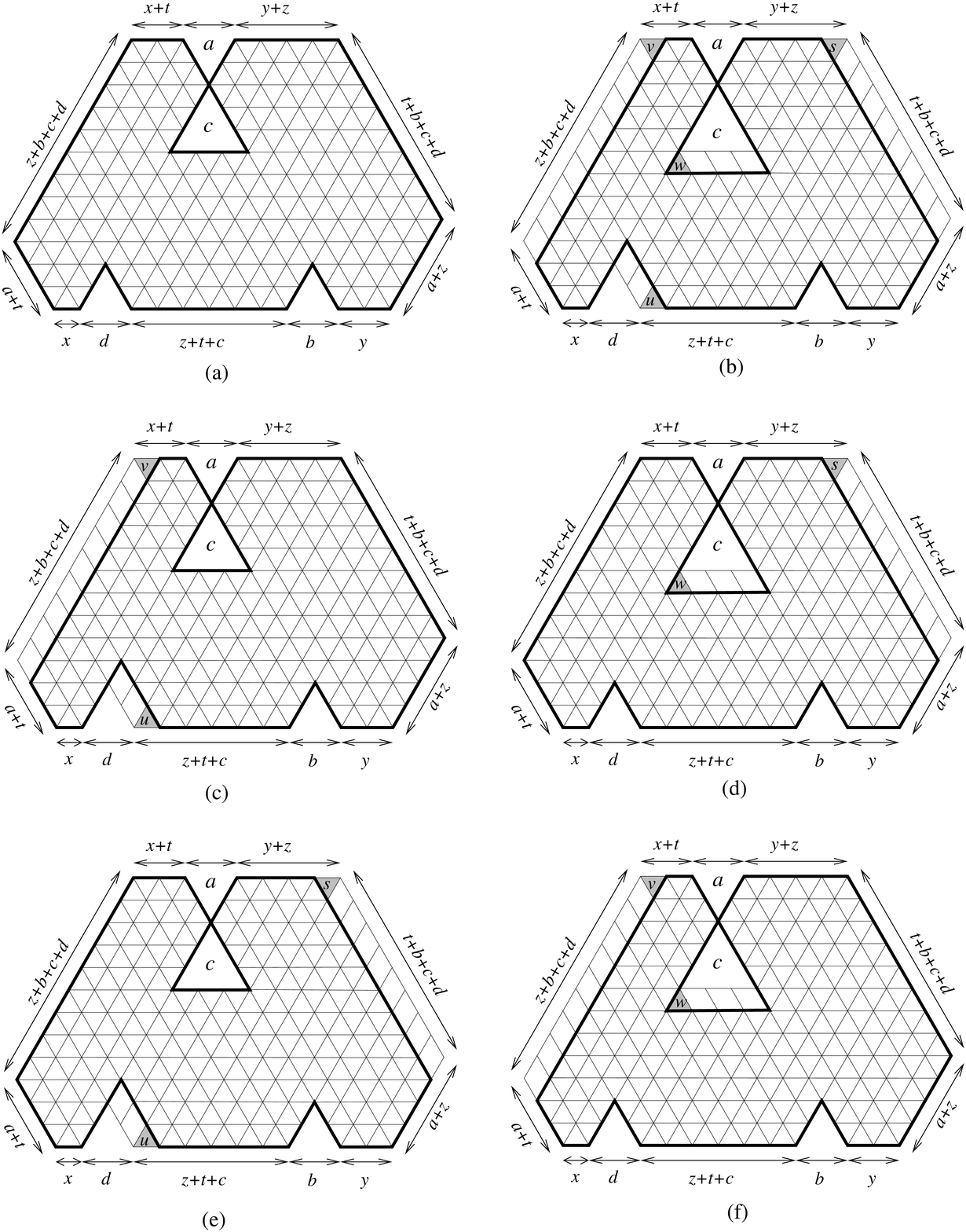}
\caption{Obtaining recurrence of the tiling generating function of $B$-regions.}
\label{Kuodentbarb}
\end{figure}

\begin{proof}
Similar to the proof of Theorem \ref{qpremain1}, we have
\begin{equation}
\frac{\M_2\left(B\begin{pmatrix}x&y&z&t\\a&b&c&d\end{pmatrix}\right)}{\sum_{\pi}q^{|\pi|}}=q^{\sum_{i=1}^{7}a_ib_iy_i+b_ia_i(a_i+1)/2},
\end{equation}
where the component box $D_i$ in $\mathcal{D}$ has size $a_i \times b_i \times c_i$ and where  the base $P_i$ of $D_i$ is $y_i$ units above  the base of the region (shown in Figure \ref{tilingdentbar}(b)), for $i=1,2,\dotsc,7$. Writing $a_i,b_i,y_i$ in terms of $a,b,c,d,x,y,z,t$ from Figures \ref{tilingdentbar}(b) and (c), we get
\begin{align}
\sum_{i=1}^{7}&a_ib_iy_i+b_ia_i(a_i+1)/2=A.
\end{align}
\normalsize Thus, we only need to prove the second equal sign in (\ref{qdentbareq}).

Next, we prove the second equal sign in (\ref{qdentbareq}) by induction on $z+t$. Our base cases are the cases when $z=0$ or $t=0$.

Assume that our region are weighted by $\wt_2$. If $z=0$, we use Region-splitting Lemma \ref{GS} to separate our region into two parts as in Figure \ref{BCdentbar}(a).  The right part is the hexagon $Hex(a,y,t+b+c+d)$ and  the left part is   $Q_{d,c}(x,t,b,a)$;  both parts are weighted by $\wt_2$. Then, we get
\begin{equation}
\M_2\left(B\begin{pmatrix}x&y&0&t\\a&b&c&d\end{pmatrix}\right)=\M_2\big(Hex(a,y,t+b+c+d)\big)\M_2\big(Q_{d,c}(x,t,b,a)\big),
\end{equation}
and the second equal sign in (\ref{qdentbareq}) follows from Corollary \ref{lem0} and Theorem \ref{qpremain1}.
If $t=0$, we also split our region into two subregions as in Figure \ref{BCdentbar}(b). The left subregion is the hexagon $Hex(z+b+c+d,x,a)$ weighted by $\wt_2$.
For the right subregion, we reflect it about a vertical line, and get the region $Q_{b,c}(y,z,d,a)$ weighted by the weight assignment in Theorem \ref{qpremain2}. Therefore, we get
\begin{equation}
\M_2\left(B\begin{pmatrix}x&y&z&0\\a&b&c&d\end{pmatrix}\right)=\M_2\big(Hex(z+b+c+d,x,a)\big)\M_3\big(Q_{b,c}(y,z,d,a)\big),
\end{equation}
and the second equal sign in (\ref{qdentbareq}) is implied by  Corollary \ref{lem0} and Theorem \ref{qpremain2}.

\medskip

For the induction step, we assume that $z,t>0$ and that (\ref{qdentbareq}) holds for any $B$-type regions having the sum of the $z$- and $t$-parameters strictly less than $z+t$.

\medskip

We apply Kuo's Theorem \ref{kuothm1} to the dual graph $G$ of  $B\begin{pmatrix}x&y&z&t\\a&b&c&d\end{pmatrix}$ weighted by $\wt_2$ as in Figure \ref{Kuodentbarb}. In particular, the four shaded unit triangles in Figure \ref{Kuodentbarb}(b) correspond to the four vertices $u,v,w,s$.  The lowest shaded unit triangle corresponds to $u$, and $v,w,s$ correspond to the next shaded unit triangles as we move clockwise from the lowest one. Figure \ref{Kuodentbarb} tells us that the product of the $M_2$-generating functions of the two regions on the top is equal to the product of the $M_2$-generating functions of the two regions in the middle, plus the product of the $M_2$-generating functions of the two regions on the bottom. Since all lozenges forced by the shaded unit triangles in Figure \ref{Kuodentbarb} have weight 1, we get
\begin{align}\label{dentbareq6}
\M_2\left(B\begin{pmatrix}x&y&z&t\\a&b&c&d\end{pmatrix}\right)&\M_2\left(B\begin{pmatrix}x&y&z-1&t-1\\a&b&c+1&d+1\end{pmatrix}\right)=\notag\\
&\M_2\left(B\begin{pmatrix}x&y&z-1&t\\a&b&c&d+1\end{pmatrix}\right)\M_2\left(B\begin{pmatrix}x&y&z&t-1\\a&b&c+1&d\end{pmatrix}\right)\notag\\
&+\M_2\left(B\begin{pmatrix}x&y&z-1&t\\a&b&c+1&d\end{pmatrix}\right)\M_2\left(B\begin{pmatrix}x&y&z&t-1\\a&b&c&d+1\end{pmatrix}\right).
\end{align}
 Finally, we only need to show that the expression after the second equal sign in (\ref{qdentbareq}) satisfies the same recurrence.

By definition of the exponent $A$, we have
\begin{align}\label{Aeq1}
A\begin{pmatrix}x&y&z&t\\a&b&c&d\end{pmatrix}&+A\begin{pmatrix}x&y&z-1&t-1\\a&b&c+1&d+1\end{pmatrix}\notag\\
&=A\begin{pmatrix}x&y&z&t-1\\a&b&c&d+1\end{pmatrix}+A\begin{pmatrix}x&y&z-1&t\\a&b&c+1&d\end{pmatrix}-d-x-t,
\end{align}
\begin{align}\label{Aeq2}
A\begin{pmatrix}x&y&z&t\\a&b&c&d\end{pmatrix}&+A\begin{pmatrix}x&y&z-1&t-1\\a&b&c+1&d+1\end{pmatrix}\notag\\
&=A\begin{pmatrix}x&y&z-1&t\\a&b&c&d+1\end{pmatrix}+A\begin{pmatrix}x&y&z&t-1\\a&b&c+1&d\end{pmatrix}.
\end{align}
\normalsize We denote by $\Phi\begin{pmatrix}x&y&z&t\\a&b&c&d\end{pmatrix}$ the expression after the second equal sign in (\ref{qdentbareq}) \textit{divided by $q^{A}$}. By (\ref{Aeq1}) and (\ref{Aeq1}), we need to show that
\small{\begin{align}\label{qdentbareq7}
q^{d+x+t}\frac{\Phi\begin{pmatrix}x&y&z&t-1\\a&b&c&d+1\end{pmatrix}}{\Phi\begin{pmatrix}x&y&z&t\\a&b&c&d\end{pmatrix}}
&\frac{\Phi\begin{pmatrix}x&y&z-1&t\\a&b&c+1&d\end{pmatrix}}{\Phi\begin{pmatrix}x&y&z-1&t-1\\a&b&c+1&d+1\end{pmatrix}}\notag\\
&+\frac{\Phi\begin{pmatrix}x&y&z&t-1\\a&b&c+1&d\end{pmatrix}}{\Phi\begin{pmatrix}x&y&z&t\\a&b&c&d\end{pmatrix}}
\frac{\Phi\begin{pmatrix}x&y&z-1&t\\a&b&c&d+1\end{pmatrix}}{\Phi\begin{pmatrix}x&y&z-1&t-1\\a&b&c+1&d+1\end{pmatrix}}=1.
\end{align}}
\normalsize
Let us simplify the first term on the left-hand side of (\ref{qdentbareq7}). We note that the numerator and the denominator of the first fraction in the first term are different only at their $d$- and $t$-parameters. By cancelling out all terms without $d$- or $t$-parameter and using the simple fact $\Hf_q(n+1)=[n]_q!\Hf_q(n)$, we can evaluate the first fraction as
\begin{align}\label{qdenbareqn1}
&\frac{\Phi\begin{pmatrix}x&y&z&t-1\\a&b&c&d+1\end{pmatrix}}{\Phi\begin{pmatrix}x&y&z&t\\a&b&c&d\end{pmatrix}}=\frac{[d]_q!}{[t-1]_q!}\frac{[a+b+c+x+y+z+2t-1]_q![a+x+t-1]_q!}{[a+b+c+d+x+z+2t-1]_q![a+c+z+t-1]_q!}\notag\\
&\times \frac{[a+c+d+x+z+2t-1]_q![a+b+c+y+2z+t-1]_q![b+c+z+t-1]_q!}{[a+b+c+d+x+y+2z+2t-1]_q![a+b+c+y+z+t-1]_q!}.
\end{align}
Similarly, the second fraction of the first term can be written as
\begin{align}\label{qdenbareqn2}
&\frac{\Phi\begin{pmatrix}x&y&z-1&t\\a&b&c+1&d\end{pmatrix}}{\Phi\begin{pmatrix}x&y&z-1&t-1\\a&b&c+1&d+1\end{pmatrix}}=\frac{[t-1]_q!}{[d]_q!}\frac{[a+b+c+x+z+2t-1]_q![a+c+z+t-1]_q!}{[a+b+c+d+x+y+z+2t-1]_q![a+x+t-1]_q!}\notag\\
&\times \frac{[a+b+c+d+x+y+2z+2t-2]_q![a+b+c+y+z+t-1]_q!}{[a+c+d+x+z+2t-1]_q![a+b+c+y+2z+t-2]_q![b+c+z+t-1]_q!}.
\end{align}
By (\ref{qdenbareqn1}) and (\ref{qdenbareqn2}), we obtain
\begin{align}
\frac{\Phi\begin{pmatrix}x&y&z&t-1\\a&b&c&d+1\end{pmatrix}}{\Phi\begin{pmatrix}x&y&z&t\\a&b&c&d\end{pmatrix}}
\frac{\Phi\begin{pmatrix}x&y&z-1&t\\a&b&c+1&d\end{pmatrix}}
{\Phi\begin{pmatrix}x&y&z-1&t-1\\a&b&c+1&d+1\end{pmatrix}}=\frac{[a+b+c+y+2z+t-1]_q}{[a+b+c+d+x+y+2z+2t-1]_q}.
\end{align}
Similarly, we can simplify the second term as
\begin{align}
\frac{\Phi\begin{pmatrix}x&y&z&t-1\\a&b&c+1&d\end{pmatrix}}{\Phi\begin{pmatrix}x&y&z&t\\a&b&c&d\end{pmatrix}}
\frac{\Phi\begin{pmatrix}x&y&z-1&t\\a&b&c&d+1\end{pmatrix}}
{\Phi\begin{pmatrix}x&y&z-1&t-1\\a&b&c+1&d+1\end{pmatrix}}=\frac{[d+x+t]_q}{[a+b+c+d+x+y+2z+2t-1]_q}.
\end{align}
Then (\ref{qdentbareq7}) becomes the following equation
\begin{equation}
\frac{q^{d+x+t}[a+b+c+y+2z+t-1]_q}{[a+b+c+d+x+y+2z+2t-1]_q}+\frac{[d+x+t]_q}{[a+b+c+d+x+y+2z+2t-1]_q}=1,
\end{equation}
which follows directly from the definition of the $q$-integer. This completes our proof.
\end{proof}

\section{Proof of Theorem \ref{qmain}}

By the equality (\ref{ratioprop}) in Proposition \ref{ratioprop}, we only need to show that

\small{\begin{align}\label{qmaineq}
&\M_2(F)=\notag\\
&=q^{\textbf{h}}\frac{\Hf_q(x)\Hf_q(y)\Hf_q(z)\Hf_q(a)^2\Hf_q(b)^2\Hf_q(c)^2\Hf_q(d)\Hf_q(e)\Hf_q(f)\Hf_q(d+e+f+x+y+z)^4}
{\Hf_q(a+d)\Hf_q(b+e)\Hf_q(c+f)\Hf_q(d+e+x+y+z)\Hf_q(e+f+x+y+z)\Hf_q(f+d+x+y+z)}\notag\\
 &\times \frac{\Hf_q(A+2x+2y+2z)\Hf_q(A+x+y+z)^2}{\Hf_q(A+2x+y+z)\Hf_q(A+x+2y+z)\Hf_q(A+x+y+2z)}\notag\\
  &\times \frac{\Hf_q(a+b+d+e+x+y+z)\Hf_q(a+c+d+f+x+y+z)\Hf_q(b+c+e+f+x+y+z)}{\Hf_q(a+d+e+f+x+y+z)^2\Hf_q(b+d+e+f+x+y+z)^2\Hf_q(c+d+e+f+x+y+z)^2}\notag\\
  &\times\frac{\Hf_q(a+d+x+y)\Hf_q(b+e+y+z)\Hf_q(c+f+z+x)}{\Hf_q(a+b+y)\Hf_q(b+c+z)\Hf_q(c+a+x)}\notag\\
  &\times \frac{\Hf_q(A-a+x+y+2z)\Hf_q(A-b+2x+y+z)\Hf_q(A-c+x+2y+z)}{\Hf_q(b+c+e+f+x+y+2z)\Hf_q(c+a+d+f+2x+y+z)\Hf_q(a+b+d+e+x+2y+z)}.
\end{align}}

\normalsize

\begin{figure}\centering
\includegraphics[width=14cm]{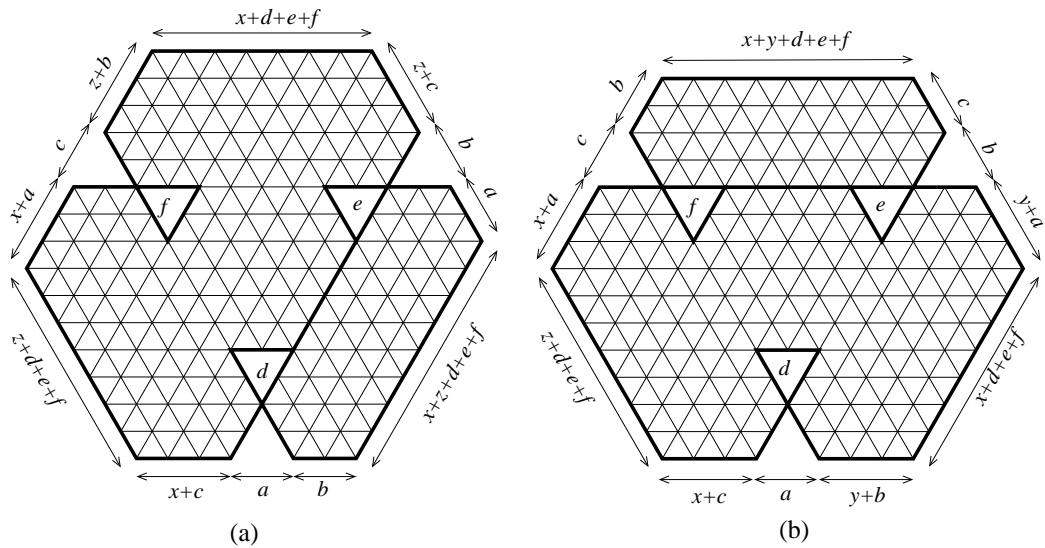}
\caption{Base cases of the Theorem \ref{main}: (a) $y=0$ and (b) $z=0$.}
\label{BC3dent}
\end{figure}

\normalsize We prove (\ref{qmaineq}) by induction on $y+z$, with the base cases are $y=0$ and $z=0$.

If $y=0$,  assume that our region is weighted by $wt_1$. We split up the region into two parts by using Region-splitting Lemma \ref{GS} as in Figure \ref{BC3dent}(a). The right part is the hexagon $Hex(x+z,d+e+f,b,a)$ weighted by $wt_1$.  For the left part, we rotate it $60^{\circ}$ clockwise, and reflect the resulting region about a vertical line. This way, we get the region $B\begin{pmatrix}b&a&x&z\\c&d&f&e\end{pmatrix}$  weighted by $\wt_2$. In particular, we get
\begin{equation}
\M_1\left(F\begin{pmatrix}x&0&z\\a&b&c\\d&e&f\end{pmatrix}\right)=\M_1\big(Hex(x+z+d+e+f,b,a)\big)\M_2\left(B\begin{pmatrix}b&a&x&z\\c&d&f&e\end{pmatrix}\right).
\end{equation}
Then (\ref{qmaineq}) follows from Proposition \ref{ratioprop}, Corollary \ref{lem0}  and Theorem \ref{qdentbarthm}.

If $z=0$, we weight our region by $wt_2$,  and split it into two parts as in Figure \ref{BC3dent}(b). Dividing the weight of each right lozenge in the upper part by $q^{x+z+a+d+e+f}$, we get the hexagon $Hex(b,x+y+d+e+f,c)$ weighted by $wt_2$. For the bottom part, we rotate it $180^{\circ}$ counter-clockwise, and get the region $B\begin{pmatrix}b&c&x&y\\a&f&d&e\end{pmatrix}$ in which a right lozenge is weighted by $q^{x+z+a+d+e+f+1-l}$, where $l$ is the distance between the top of the lozenge and the bottom of the $B$-type region (and all left and vertical lozenges are weighted by $1$ as usual). Dividing the weight of each right lozenge by $q^{x+z+a+d+e+f+1}$, we get back the weight $\wt_2$, \textit{where $q$ is replaced by $q^{-1}$}. By Corollary \ref{lem0}, Theorem \ref{qdentbarthm} and the simple fact $[n]_{q^{-1}}=[n]_q/q^{n-1}$, we obtain (\ref{qmaineq}).

\medskip

For induction step, we assume that $y$ and $z$ are positive, and that (\ref{qmaineq}) holds for any $F$-type region having the sum of $y$- and $z$-parameters strictly less than $y+z$.

Apply Kuo's Theorem \ref{kuothm2} to the dual graph $G$ of the region $R$ weighted by $\wt_2$, where $R$ is the region obtained from $F\begin{pmatrix}x&y&z\\a&b&c\\d&e&f\end{pmatrix}$ by adding a band of $2(z+d+r+f)-1$ unit triangles along the southwest side (see Figure \ref{Kuo3dent}; the four vertices $u,v,w,s$ correspond to the four shaded unit triangles).  By removing lozenges forced by the shaded unit triangles as in Figure \ref{Kuo3dentb}, we get a new $F$-type region having the same $M_2$-generating function (since all the forced lozenges have weight $1$).
Theorem \ref{kuothm2} and  Figure \ref{Kuo3dentb} gives  us the following recurrence:
\small{\begin{align}\label{3denteq1}
\M_2\left(F\begin{pmatrix}x&y&z\\a&b&c\\d&e&f\end{pmatrix}\right)&\M_2\left(F\begin{pmatrix}x+1&y-1&z-1\\a&b&c\\d&e+1&f\end{pmatrix}\right)=\notag\\
&\M_2\left(F\begin{pmatrix}x+1&y&z-1\\a&b&c\\d&e&f\end{pmatrix}\right)\M_2\left(F\begin{pmatrix}x&y-1&z\\a&b&c\\d&e+1&f\end{pmatrix}\right)\notag\\
&+\M_2\left(F\begin{pmatrix}x+1&y-1&z\\a&b&c\\d&e&f\end{pmatrix}\right)\M_2\left(F\begin{pmatrix}x&y&z-1\\a&b&c\\d&e+1&f\end{pmatrix}\right).
\end{align}}
\normalsize 
Our next job is checking that the expression on the right-hand side of (\ref{qmaineq}) satisfies the same recurrence.

\begin{figure}\centering
\includegraphics[width=8cm]{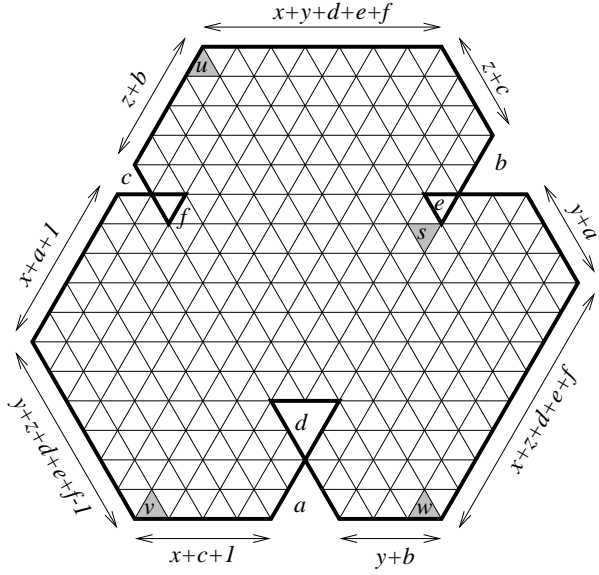}
\caption{How we apply Kuo condensation to a hexagon with three dents.}
\label{Kuo3dent}
\end{figure}
\begin{figure}\centering
\includegraphics[width=14cm]{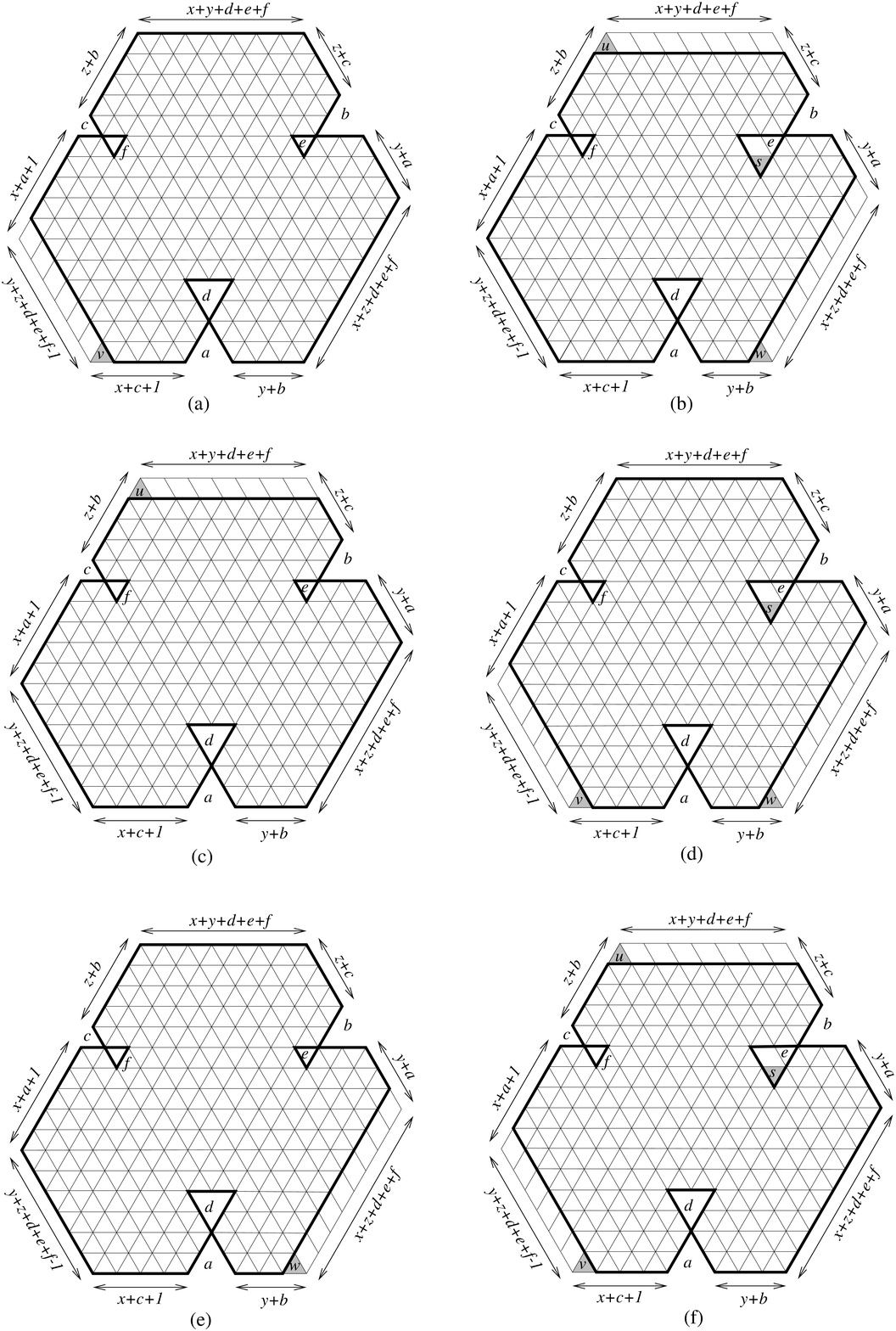}
\caption{Obtaining recurrence of the $M_2$-generating functions of tilings of $F$-type regions.}
\label{Kuo3dentb}
\end{figure}

By definition of the exponent $\textbf{h}$, we have
\small{\begin{align}
\textbf{h}\begin{pmatrix}x+1&y&z-1\\a&b&c\\d&e&f\end{pmatrix}+\textbf{h}\begin{pmatrix}x&y-1&z\\a&b&c\\d&e+1&f\end{pmatrix}=
\textbf{h}\begin{pmatrix}x&y&z\\a&b&c\\d&e&f\end{pmatrix}+\textbf{h}\begin{pmatrix}x+1&y-1&z-1\\a&b&c\\d&e+1&f\end{pmatrix}
\end{align}}
\normalsize and
\small{\begin{align}
\textbf{h}\begin{pmatrix}x+1&y-1&z\\a&b&c\\d&e&f\end{pmatrix}+&\textbf{h}\begin{pmatrix}x&y&z-1\\a&b&c\\d&e+1&f\end{pmatrix}=\notag\\
&a+d+x+y+\textbf{h}\begin{pmatrix}x&y&z\\a&b&c\\d&e&f\end{pmatrix}+\textbf{h}\begin{pmatrix}x+1&y-1&z-1\\a&b&c\\d&e+1&f\end{pmatrix}.
\end{align}}
\normalsize Denote by $\Psi\begin{pmatrix}x&y&z\\a&b&c\\d&e&f\end{pmatrix}$ the expression on the right-hand side of (\ref{qmaineq})  \textit{divided by $q^{\textbf{h}}$} (i.e. the expression on the right-hand side of the equation (\ref{maineqbox})). We need to show that
\small{\begin{align}\label{qcheck2}
\frac{\Psi\begin{pmatrix}x+1&y&z-1\\a&b&c\\d&e&f\end{pmatrix}}{\Psi\begin{pmatrix}x&y&z\\a&b&c\\d&e&f\end{pmatrix}}
&\frac{\Psi\begin{pmatrix}x&y-1&z\\a&b&c\\d&e+1&f\end{pmatrix}}{\Psi\begin{pmatrix}x+1&y-1&z-1\\a&b&c\\d&e+1&f\end{pmatrix}}\notag\\
&+q^{a+d+x+y}\frac{\Psi\begin{pmatrix}x+1&y-1&z\\a&b&c\\d&e&f\end{pmatrix}}{\Psi\begin{pmatrix}x&y&z\\a&b&c\\d&e&f\end{pmatrix}}
\frac{\Psi\begin{pmatrix}x&y&z-1\\a&b&c\\d&e+1&f\end{pmatrix}}{\Psi\begin{pmatrix}x+1&y-1&z-1\\a&b&c\\d&e+1&f\end{pmatrix}}=1.
\end{align}}
\normalsize We will simplify the two terms on the left-hand side of (\ref{qcheck2}). First, we evaluate the first fraction in the first term with the notice that its numerator and denominator are different only at $x$- and $t$-parameters. Cancelling out all terms without $x$- or $t$-parameter and using the trivial fact $\Hf_q(n+1)=[n]_q!\Hf_q(n)$, we get
\small{\begin{align}\label{3denteq3}
&\frac{\Psi\begin{pmatrix}x+1&y&z-1\\a&b&c\\d&e&f\end{pmatrix}}{\Psi\begin{pmatrix}x&y&z\\a&b&c\\d&e&f\end{pmatrix}}=\frac{[x]_q!}{[z-1]_q!}\frac{[A+x+y+2z-1]_q!}{[A+2x+y+z]_q!}\frac{[a+d+x+y]_q![b+c+z-1]_q!}{[b+e+y+z]_q![a+c+x]_q!}\notag\\
&\times\frac{[A-b+2x+y+z]_q![b+c+e+f+x+y+z-1]_q!}{[A-a+x+y+2z-1]_q![a+c+d+f+2x+y+z]_q!}.
\end{align}}
\normalsize Similarly, one can evaluate the second fraction of the first term as
\small{\begin{align}\label{3denteq4}
&\frac{\Psi\begin{pmatrix}x&y-1&z\\a&b&c\\d&e+1&f\end{pmatrix}}{\Psi\begin{pmatrix}x+1&y-1&z-1\\a&b&c\\d&e+1&f\end{pmatrix}}=\frac{[z-1]_q!}{[x]_q!}\frac{[A+2x+y+z]_q!}{[A+x+y+2z-1]_q!}\frac{[b+e+y+z-1]_q![a+c+x]_q!}{[a+d+x+y-1]_q![b+c+z-1]_q!}\notag\\
&\times\frac{[A-a+x+y+2z-1]_q![a+c+d+f+2x+y+z-1]_q!}{[A-b+2x+y+z]_q![b+c+e+f+x+y+z-1]_q!}.
\end{align}}
\normalsize By (\ref{3denteq3}) and (\ref{3denteq4}), we get the first term simplified as
\small{\begin{equation}
\frac{\Psi\begin{pmatrix}x+1&y&z-1\\a&b&c\\d&e&f\end{pmatrix}}{\Psi\begin{pmatrix}x&y&z\\a&b&c\\d&e&f\end{pmatrix}}
\frac{\Psi\begin{pmatrix}x&y-1&z\\a&b&c\\d&e+1&f\end{pmatrix}}{\Psi\begin{pmatrix}x+1&y-1&z-1\\a&b&c\\d&e+1&f\end{pmatrix}}=\frac{[a+d+x+y]_q}{[a+c+d+f+2x+y+z]_q}.
\end{equation}}
\normalsize Similarly, we have
\small{\begin{equation}
\frac{\Psi\begin{pmatrix}x+1&y-1&z\\a&b&c\\d&e&f\end{pmatrix}}{\Psi\begin{pmatrix}x&y&z\\a&b&c\\d&e&f\end{pmatrix}}
\frac{\Psi\begin{pmatrix}x&y&z-1\\a&b&c\\d&e+1&f\end{pmatrix}}{\Psi\begin{pmatrix}x+1&y-1&z-1\\a&b&c\\d&e+1&f\end{pmatrix}}=\frac{[c+f+x+z]_q}{[a+c+d+f+2x+y+z]_q}.
\end{equation}}
\normalsize Therefore, the equation (\ref{qcheck2}) is equivalent to the following equation
\begin{equation}
\frac{[a+d+x+y]_q}{[a+c+d+f+2x+y+z]_q}+\frac{q^{a+d+x+y}[c+f+x+z]_q}{[a+c+d+f+2x+y+z]_q}=1,
\end{equation}
which is obviously true by the definition of the $q$-integer. This finishes our proof.

\section{Plane partitions with constraints}

In this section, we present a consequence of  Theorem \ref{main} on enumeration of (ordinary) plane partitions with certain constraints.

\begin{figure}\centering
\includegraphics[width=14cm]{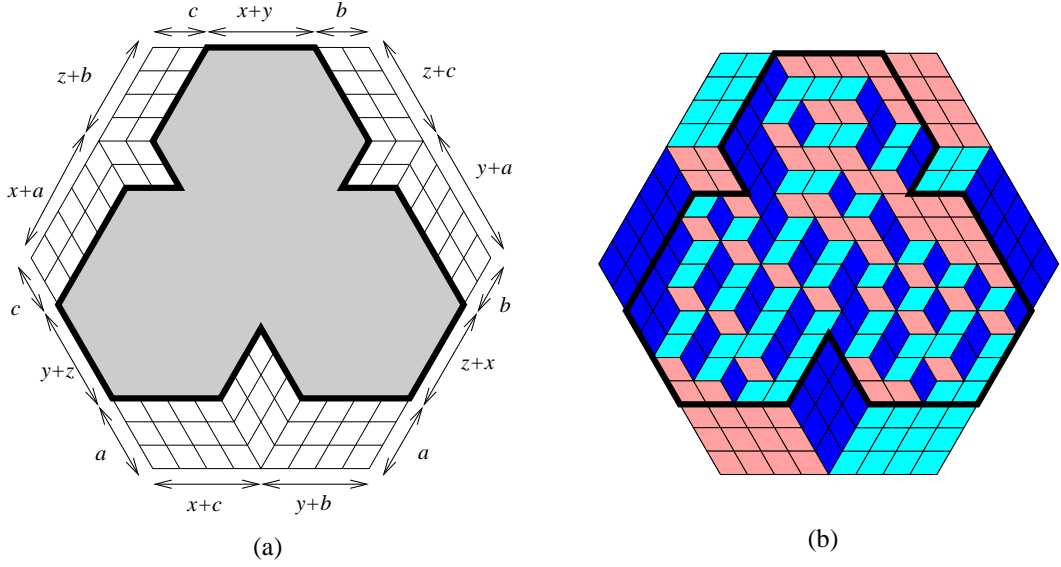}%
\caption{(a) Adding forced lozenges along the northeast, the northwest and the south sides of the region $N_{3,2,2}(2,2,2)$ (the shaded region restricted by the bold contour) (b) Viewing a the lozenge tiling  of $N_{3,2,2}(2,2,2)$, where forced lozenges have been added, as a pile of unit cubes fitting in a $ 9\times8\times 9$ box.}
\label{PPconstraint}
\end{figure}

Let us assume that the $d$-, $e$- and $f$-parameters in our $F$-type region are equal to $0$.  The region $N_{a,b,c}(x,y,z):=F\begin{pmatrix}x&y&z\\a&b&c\\0&0&0\end{pmatrix}$ becomes a hexagon with three triangular dents on three non-consecutive sides.  By adding the forced lozenges as in Figure \ref{PPconstraint}(a) to each tiling of $N_{a,b,c}(x,y,z)$, we get a lozenge tiling of the semi-regular hexagon $Hex(z+x+a+b,x+y+b+c,y+z+c+a)$. The latter lozenge tiling in turn corresponds to a plane partition (in the tabular form) having $z+x+a+b$ rows and $x+y+b+c$ columns with entries at most $y+z+c+a$ (or, a pile of unit cubes fitting in a $(z+x+a+b)\times(x+y+b+c)\times(y+z+c+a)$ box as in  Figure \ref{PPconstraint}(b)). The forced lozenges give certain constraints on this plane partition. In particular, the forced lozenges on the northwest side imply that:
\begin{enumerate}
\item[(i).] \textit{The first $z+b$ entries of the columns $1,2,\dotsc,c$ are all $y+z+c+a$. Moreover, the remaining entries in these columns are at most $y+z+a$.}
\end{enumerate}
The forced lozenges on the northeast side say:
\begin{enumerate}
\item[(ii).]  \textit{The last $a$ entries of the rows $1,2,\dotsc, b$ are all $y+a$.}
\end{enumerate}
Finally, the forced lozenges along the south side is equivalent to the fact that:
\begin{enumerate}
\item[(iii).] \textit{The last $y+b$ entries of the rows $z+x+b+1,z+x+b+2,\dotsc,z+x+b+a$ are all $0$; and the remaining entries of these rows are at least $a$.}
\end{enumerate}


Thus by Theorem \ref{main}, we get the following enumeration.

\begin{cor} Let $a,b,c,x,y,z$ be non-negative integers. The number of plane partitions having $z+x+a+b$ rows and $x+y+b+c$ columns with entries at most $y+z+c+a$, where the constraints $(i)$, $(ii)$ and $(iii)$ hold, is equal to
\begin{align}
 &\Hf(x)\Hf(y)\Hf(z)\Hf(a)\Hf(b)\Hf(c)\Hf(x+y+z)\notag\\
 &\times \frac{\Hf(a+b+c+2x+2y+2z)\Hf(a+b+c+x+y+z)^2}{\Hf(a+b+c+2x+y+z)\Hf(a+b+c+x+2y+z)\Hf(a+b+c+x+y+2z)}\notag\\
  &\times \frac{\Hf(a+b+x+y+z)\Hf(a+c+x+y+z)\Hf(b+c+x+y+z)}{\Hf(a+x+y+z)^2\Hf(b+x+y+z)^2\Hf(c+x+y+z)^2}\notag\\
  &\times\frac{\Hf(a+x+y)\Hf(b+y+z)\Hf(c+z+x)}{\Hf(a+b+y)\Hf(b+c+z)\Hf(c+a+x)}.
\end{align}
\end{cor}

\section{Concluding remarks}
This paper gives an instance of  the great power of Kuo condensation in $q$-enumeration of lozenge tilings of hexagons with dents. The author has realized that there are many more `nice' $q$-enumerations that can be proven by the same method (details will appear in a separate paper).

Moreover, Kuo condensation can be used to give direct proofs for Corollary \ref{lem0} and Lemma \ref{qlem1}.



Finally, the author wants to thank David Wilson for introducing the q-mode of his software, \texttt{vaxmacs 1.6e} (available for downloading
at \url{http://dbwilson.com/vaxmacs/}), which is really helpful in verifying the formulas in this paper.

%
%

\end{document}